\documentclass[11pt]{article}
 \usepackage{multirow}
\usepackage{graphicx}
\usepackage{subfigure}
\usepackage{epstopdf}
\usepackage{mathrsfs}
\usepackage{bbm}
\usepackage{booktabs}
\usepackage{bm}
\usepackage{extarrows}
\usepackage{float}
\usepackage{amsfonts}
\usepackage{amssymb}
\usepackage{amsmath,amsthm}
\usepackage{indentfirst}
\usepackage{hyperref}
\usepackage{booktabs}
\usepackage{array}
\usepackage{caption}
\usepackage{algorithm}
\usepackage{algorithmic}
\usepackage{authblk}
\usepackage{cite}
\usepackage{color}
\usepackage{mdwlist}
\usepackage{stfloats}
\usepackage[marginal]{footmisc}
\numberwithin{equation}{section}

\usepackage{cases}
\def\opn#1#2{\def#1{\operatorname{#2}}} % to make operators
\opn\chara{char} \opn\length{\ell}
\allowdisplaybreaks
 \opn\projdim{proj\,dim} \opn\injdim{inj\,dim}
\opn\rank{rank} \opn\depth{depth} \opn\grade{grade}
\opn\height{height} \opn\embdim{emb\,dim} \opn\codim{codim}

\opn\Tr{Tr} \opn\bigrank{big\,rank}
\opn\superheight{superheight}\opn\lcm{lcm}
\opn\trdeg{tr\,deg}%
\opn\reg{reg} \opn\lreg{lreg}
\opn\Ker{Ker} \opn\Coker{Coker} \opn\Im{Im} \opn\Hom{Hom}
\opn\Tor{Tor} \opn\Ext{Ext} \opn\End{End} \opn\Aut{Aut} \opn\id{id}

\opn\nat{nat}
\opn\pff{pf}%   \pf exists already
\opn\Pf{Pf} \opn\GL{GL} \opn\SL{SL} \opn\mod{mod} \opn\ord{ord}
\def\Implies{\ifmmode\Longrightarrow \else
     \unskip${}\Longrightarrow{}$\ignorespaces\fi}
\def\implies{\ifmmode\Rightarrow \else
     \unskip${}\Rightarrow{}$\ignorespaces\fi}
\def\iff{\ifmmode\Longleftrightarrow \else
     \unskip${}\Longleftrightarrow{}$\ignorespaces\fi}

\let\:=\colon

\newtheorem{Theorem}{Theorem}[section]
\newtheorem{thm}{Theorem}[section]

\newtheorem{lem}[Theorem]{Lemma}

\newtheorem{rem}[Theorem]{Remark}

\newtheorem{ex}[Theorem]{Example}

\newtheorem{define}[Theorem]{Definition}

\let\epsilon=\varepsilon
\let\kappa=\varkappa
\usepackage{geometry}
\opn\ini{in} \opn\inm{inm} \opn\Sym{Sym} \opn\diag{diag}
\opn\Ii{(i)} \opn\Iii{(ii)}

\begin{document}
	\title{  On spectral Petrov-Galerkin method for solving fractional initial value problems in weighted Sobolev space}

	\author[a]{Shengyue Li \footnote{ lishengyue@seu.edu.cn} }
	\author[a]{Wanrong Cao \thanks{Corresponding author: wrcao@seu.edu.cn} ~}
	\affil[a]{School of Mathematics,
			Southeast University, Nanjing 210096, P.R. China.}
	\author[b]{ \ Zhaopeng Hao \footnote{hao27@purdue.edu} }
	\affil[b]{ Department of Mathematics, Purdue University, West Lafayette, IN 47906, USA.}
	\renewcommand\Authands{ \ and}

	\date{}
	\maketitle
	\noindent \textbf{Abstract}: 
		In this paper, we investigate a spectral Petrov-Galerkin method for fractional initial value problems. Singularities of the solution at the origin inherited from the weakly singular kernel of the fractional derivative are considered, and the regularity is constructed for the solution in weighted Sobolev space. We present an optimal error estimate of the spectral Petrov-Galerkin method, and prove that the convergence order of the method in the weighted $L^2$-norm is $3\alpha+1$ for smooth source term, where $\alpha$ is the order of the fractional derivative.
		An iteration algorithm with a quasi-linear complexity is considered to solve the produced linear system. Numerical experiments verify the theoretical findings and show the efficiency of the proposed algorithm, and exhibit that the presented numerical method works well for some time-fractional diffusion equations after suitable temporal semi-discrete.
	
	\ 
	
	\noindent \textbf{Keywords}: fractional derivative singularity, weighted Sobolev space, spectral method, optimal error estimates

\section{Introduction}
Time-fractional differential equations (TFDEs) have attracted significant attention due to the capability of modeling some long-time memory complex systems and anomalous diffusion problems \cite{car,nig,hat,per}. With the increasing application, numerical methods have also been well studied for TFDEs, such as  finite difference schemes \cite{ya,cao1,hu,lub1}, finite element methods \cite{xie1,must,mcl}, spectral methods \cite{xu1,zhang1,chen2}, etc.  The model TFDE has the form
\begin{eqnarray}\left\{\label{eq:TFDE}\begin{aligned}
		&{}_{0}D_t^{\alpha}u+Au=g(x,t),&x\in \Omega,\ t\in I,\quad\quad\ \\
		&u(x,0)=0,&x\in \Omega,\quad \quad \quad \quad\quad\  \\
		&u=0, &x \in \partial \Omega,\ t\in I,\ \ \ \ \ 
	\end{aligned}\right.\end{eqnarray}
where  $\Omega\subset\mathbb{R}$ is an open interval, the linear operator $A$ on $\Omega$ is positive definite,  $g$ is a given source term,  and ${}_{0}D_t^{\alpha}$ is the  Caputo fractional derivative with $0<\alpha<1$, defined by
\begin{equation*}
	{}_{0}D_t^{\alpha}y(t)=\frac{1}{\Gamma(1-\alpha)}\int_{0}^{t}\frac{y'(s)}{(t-s)^{\alpha}}ds.
\end{equation*} 
By some suitable space discretization, the TFDE \eqref{eq:TFDE} can be converted a fractional initial value problem (FIVP).  In Example \ref{exadd} of Section \ref{section6}, we give the details of solving \eqref{eq:TFDE} by applying the proposed algorithm in this framework, where $-A$ is the Laplacian operator.

As a simple but widely occurring model, the FIVP has been well studied in standard Sobolev space \cite{chen1,Dieth1,xu1,xie1,zhang1}. The exact solution of fractional differential equations (FDEs) inherits singularities caused by the weakly singular kernel in fractional derivative operators, so the regularity will be very low in standard Sobolev space even for smooth data. The fact has been revealed in \cite{chen1,jin1,martin1,wang1}, which motivated the use of weighted Sobolev space to better incorporate singularities at the endpoints.

Consider the following fractional initial value problem (FIVP)
\begin{eqnarray}\label{FIVP}\begin{aligned}
		{}_{0}D_t^{\alpha}u+\lambda u&=f(t),\ t\in I=(0,T],\\
		u(0)&=0,
\end{aligned}\end{eqnarray}
where $\lambda$ is a nonnegative constant,  $f(t)$ is a given function.  In spite of the available explicit form of the solution to FIVP \eqref{FIVP}, which has been given in \cite{Dieth1}, the  exact solution expressed by Mittag-Leffler functions is hard to use directly, either not trivial to  observe its regularity. 
In this paper, we analyze the regularity of the solution  in weighted Sobolev space, and investigate the convergence order of a spectral Petrov-Galerkin method based on the regularity results.

For fractional boundary value problems, the regularity of solutions has been well investigated in weighted Sobolev space, see  \cite{acosta1,acosta2,ervin2,hao1,hao2,hao3,zhang2}. For FIVP \eqref{FIVP} with $\lambda=0$, Zhang \cite{zhang1} has considered its regularity in weighted Sobolev space, and presented an optimal error estimates of spectral Petrov-Galerkin and collocation methods. However, the semi-discrete in temperal of the TFDE \eqref{eq:TFDE} leads to \eqref{FIVP} with $\lambda\neq 0$. Actually, this more general case can provide a framework for dealing with some diverse fractional problems, which appear in plenty of models,  e.g.,  anomalous diffusion model  on fractals \cite{nig}, options pricing model in financial market \cite{car}, viscoelastic models in blood flow \cite{per}, etc. Hence, it is meaningful to reveal the regularity of solution to the FIVP \eqref{FIVP} in weighted Sobolev space and to design its efficient numerical approximation.

In this work, we present a full regularity analysis in weighted Sobolev space for the FIVP \eqref{FIVP} by using a bootstrapping technique. Specifically,  we prove that the regularity index for $t^{-\alpha}u$ is $\alpha+\min\{r,2\alpha+1-\epsilon\}$ when $\lambda>0$, where $\epsilon>0$ is an arbitrary small number and $r\geq 0$ is regularity index of $f$ in a weighted Sobolev space, see Theorem \ref{theorem3.6}. Moreover, when the source term $f\in C^1[0,T]$ and vanishes at the origin, we analyze the regularity of the solution by converting the original problem to a weakly singular Volterra integral equation and presenting the expression of the solution. It indicates that the regularity indexes of  $t^{-\alpha}u$  respectively are  $3\alpha+1-\epsilon$ and $3\alpha+3-\epsilon$  for smooth enough source term $f$ with $f(0)\not=0$ and $f(0)=0$.
%Moreover, for the source term with weak singularity at the origin, we show that the regularity index of  $t^{-\beta}u$ is  $r+\alpha$, where $r$ is the regularity index of $t^{\alpha-\beta}f$ in weighted Sobolev space with negative parameter $\beta$, which is elaborated in  Theorem \ref{theorem3.7}. And numerical results verify that the numerical solution of the spectral Petrov-Galerkin method using the weighted basis \eqref{space} with negative parameter are more accurate than with positive parameter in this case, see Example \ref{example5.3}.

%Compare to those with positive parameter, numerical results (Example \ref{example5.3}) indicate that higher accuracy can be observe with negative parameter.
%with regularity index $r$
% From the weakly singular kernel of fractional derivative, solutions of FDEs naturally inherit weak singularity.
%%在加权空间上的正则性分析

%%写关于谱方法的
In order to avoid the loss of accuracy due to the singularity near the initial or boundary, some advances have been made among the numerical community.  Nonuniform grids have been used to keep errors small near the singularity  \cite{ya,zhao,must}; Correction terms were adopted to present algorithms with globally high-order convergence \cite{lub,cao1,zeng3}; Non-polynomial basis functions have been employed to compensate for the weakly singular behavior of solutions at the endpoints and  enhance the accuracy of numerical methods \cite{tang1,zeng1,chen2}.

Over the past two decades,   spectral methods based on the nonpolynomial basis functions were proposed for solving  fractional model.  In \cite{zayer3} and \cite{zayer2}, Zayernouri and Karniadakis  developed an exponentially accurate fractional spectral collocation method with poly-fractionomials for solving steady-state and time-dependent FDEs. In \cite{chen1}, Chen, Shen, and Wang studied approximation properties of generalized Jacobi polynomials in weighted Sobolev space and developed a spectral Petrov-Galerkin method for FIVPs without reaction terms. Recently, in \cite{chen2}, the authors developed a spectral Galerkin method  in time with the log orthogonal functions for solving subdiffusion equations which shows the spectral accuracy, provided some assumptions on the smoothness of data.

In the current work,
we employ the weighted basis $t^{\alpha}P_N(I)$ to approximate the solution of FIVP \eqref{FIVP} in the framework of Petrov-Galerkin method, where $P_N(I)$ is the set of algebraic polynomials of order up to $N$. Though the idea has been widely used to construct spectral methods \cite{zhang1,chen1,zeng1} for FDEs to recover accuracy from the endpoint singularity, the occurrence of the reaction term makes the error estimate in weighted Sobolev space more challenging and complicated. Thanks to the idea of introducing an ultra-weak formulation for fractional elliptic equations \cite{hao2}, we consider also a weak Petrov-Galerkin formulation of \eqref{FIVP} and  are able to prove the regularity in  weighted Sobolev space when  the regularity index of $f$ is $r\geq -\alpha$. Based on this formulation, we present an optimal error estimate for the spectral Petrov-Galerkin method.
For smooth enough source term $f$, the convergence order, in weighted $L^2$-norm,  is $3\alpha+3-\epsilon$ if $f(0)=0$, and $3\alpha+1-\epsilon$ if $f(0)\neq 0$, while for relatively rough $f$, the order of $r+\alpha$ can be obtained, see Theorem \ref{conver}.

In summary, the main contributions of this work lie in
\vspace{-\topsep}
\begin{itemize*}
	\item giving a full regularity analysis of FIVP \eqref{FIVP} in weighted Sobolev space for both smooth and rough data,
	\item proving optimal error estimates for the spectral Petrov-Galerkin method without any regularity assumption on the analytical solution,
	\item providing a framework to solve TFDEs after suitable semi-discrete in space.
\end{itemize*}
\vspace{-\topsep}
To the best of our knowledge, we are not aware of any other work  presenting  optimal error estimates of  spectral Petrov-Galerkin methods for the FIVP \eqref{FIVP} with $\lambda>0$. Here error estimates are consistent with the regularity results and numerical experiments confirm these theoretical predictions.

The remainder of this paper is outlined as follows. In Section 2, we introduce Jacobi polynomials and fractional Sobolev spaces. The preliminary definitions and necessary lemmas are also given in this section. For the regularity index $r\geq 0$, the regularity of exact solution of the FIVP \eqref{FIVP} is studied in weighted Sobolev space in Section \ref{section3}. In Section \ref{section4}, the well-posedness and regularity of solution of the weak formulation are presented in weighted Sobolev space for $r\geq -\alpha$. Based on the regularity analysis, the spectral Petrov-Galerkin method is presented and its optimal error estimate is given  in section \ref{section5}. In section \ref{section6}, we present both direct and iteration solvers and give several numerical experiments  to verify the theoretical findings, where it is observed that the convergence order and accuracy of the numerical solutions in standard $L^2$-norm are higher than its in weighted $L^2$-norm, and the presented iteration solver is efficient; see Example \ref{ex1}.

\section{Preliminary}\label{section2}
\setcounter{equation}{0}
In this section, we  recall some necessary notations and definitions of fractional derivatives, Jacobi polynomials,  and  Sobolev spaces to be used later.

\begin{define}[\cite{samko1}]
	For $0<\alpha<1$, the left and right fractional integrals are defined, respectively, as
	\begin{align}\label{fi}
		{}_{0}I_t^{\alpha}u(t)&=\frac{1}{\Gamma(\alpha)}\int_0^{t}\frac{u(s)}{(t-s)^{1-\alpha}}ds,\ t>0, \nonumber\\
		{}_{t}I_T^{\alpha}u(t)&=\frac{1}{\Gamma(\alpha)}\int_t^{T}\frac{u(s)}{(s-t)^{1-\alpha}}ds,\ t<T.\nonumber
	\end{align}
	For $0<\alpha<1$,  the left and right Caputo fractional derivatives are defined by
	\begin{equation*}\label{rlfd}
		{}_{0}D_{t}^{\alpha}u(t)={}_{0}I_t^{1-\alpha}( D u(t)),\ t>0, \ {}_{t}D_{T}^{\alpha}u(t)=-{}_{t}I_T^{1-\alpha} ( Du(t)),\ t<T,
	\end{equation*}
	and the left and right Riemann-Liouville fractional derivatives are defined by
	\begin{equation*}\label{rlfd}
		{}^{RL}_{\ \ \! 0}\!D_{t}^{\alpha}u(t)=D( {}_{0}I_t^{1-\alpha}u(t)),\ t>0, \ {}^{RL}_{\ \  t}\!D_{T}^{\alpha}u(t)=-D( {}_{t}I_T^{1-\alpha}u(t)), \ t<T.
	\end{equation*}
\end{define}
The two definitions are linked by the following relationships
\begin{equation}\label{relationship}
	{}^{RL}_{\ \ \! 0}\!D_{t}^{\alpha}u(t)=\frac{u(0)}{\Gamma(1-\alpha)t^{\alpha}}+{}_{0}D_{t}^{\alpha}u(t),\ \ {}^{RL}_{\ \  t}\!D_{T}^{\alpha}u(t)=\frac{u(T)}{\Gamma(1-\alpha)(T-t)^{\alpha}}+{}_{t}D_{T}^{\alpha}u(t).
\end{equation}
Note that, by virtue of \eqref{relationship}, as the homogeneous condition considered in \eqref{FIVP}, the Caputo definition coincides with the Riemann-Liouville version. Hence, many useful tools established by using Riemann-Liouville derivatives can be applied here.

According to the inverse property in \cite {Dieth1}, for any absolutely integrable function $v$ and real $\alpha\geq0$,
\begin{eqnarray}\label{inverse}
	{}_{0}D_t^{\alpha}{}_{0}I_t^{\alpha}v(t)=v(t), \mbox{ a.e. in I}.
\end{eqnarray}

\begin{lem}[\cite{zhang1}]\label{parts}
	Supposing that $u\in H^{\alpha}(I)$, $0<\alpha<1$, and $v\in H^{\alpha_2}(I)$, it holds that
	\begin{eqnarray}\label{pars}
		({}_{0}D_{t}^{\alpha}u,v)=({}_{0}D_{t}^{\alpha_1}u,{}_{t}D_{T}^{\alpha_2}v),
	\end{eqnarray}
	where $\alpha_1,\alpha_2\geq0$, and $\alpha_1+\alpha_2=\alpha$.
\end{lem}

\begin{lem}[ \cite{samko1}]\label{fparts}
	If $\phi(x)\in L^p(I)$ and $\psi(x)\in L^q(I)$, $p,q\geq1$, we have
	\begin{eqnarray}\label{fpars}
		({}_{0}I_t^{\alpha}\phi,\psi)=(\phi,{}_{t}I_T^{\alpha}\psi),
	\end{eqnarray}
	where $1/p+1/q=1+\alpha$ $(p,q>1)$ or $1/p+1/q<1+\alpha$.
\end{lem}

\subsection*{\bf{Jacobi polynomials}}

For $\gamma,\beta>-1$, $n\in N$,  $x\in[-1,1]$,  $P_n^{\gamma,\beta}(x)$ is the classical Jacobi polynomial \cite{szego1} of degree $n$. Let $x=\frac{2t}{T}-1$, we transform the domain of the family of Jacobi polynomials $P_n^{\gamma,\beta}(x)$ to $[0,T]$ and introduce $$Q_{n}^{\gamma,\beta}(t)=P_n^{\gamma,\beta}(\frac{2t}{T}-1),\;t\in[0,T].$$

\begin{itemize*} 
	\item{\bf{Orthogonality.}}
	The Jacobi polynomials $P_n^{\gamma,\beta}$ are mutually orthogonal: for $\gamma,\beta>-1$,
	\begin{eqnarray}\label{jacobi}
		\int_{-1}^1(1-x)^{\gamma}(1+x)^{\beta}P_m^{\gamma,\beta}P_n^{\gamma,\beta}dx=\delta_{mn} |\|P_n^{\gamma,\beta}\||^2,
	\end{eqnarray}
	where $\delta_{mn}$ is the Kronecker function and
	\begin{eqnarray*}
		|\|P_n^{\gamma,\beta}\||^2
		=\frac{2^{\gamma+\beta+1}}{2n+\gamma+\beta+1}\frac{\Gamma(n+\gamma+1)\Gamma(n+\beta+1)}{\Gamma(n+1)\Gamma(n+ \gamma + \beta+1)}.
	\end{eqnarray*}

	\item{\bf{Fractional integral.}}
	For $\gamma\in R$, $\beta>-1$, the fractional integral of the weighted Jacobi polynomial \cite{askey1} is
	$${}_{-1}I_x^{\alpha}((1+x)^{\beta}P_n^{\gamma,\beta}(x))=\frac{\Gamma(n+\beta+1)}{\Gamma(n+\beta+\alpha+1)}(1+x)^{\beta+\alpha}P_n^{\gamma-\alpha,\beta+\alpha}(x),$$
	Hence,
	\begin{align}
		{}_{0}I_t^{\alpha}(t^{\beta}Q_n^{\gamma,\beta}(t))&=\frac{1}{\Gamma(\alpha)}\int_0^t\frac{s^\beta P_n^{\gamma,\beta}(\frac{2s}{T}-1)}{(t-s)^{1-\alpha}}ds\nonumber\\
		&\xlongequal[]{\text{Let $s=\frac{(r+1)T}{2}$}}\frac{1}{\Gamma(\alpha)}\int_{-1}^{\frac{2t}{T}-1}\frac{\frac{(r+1)^\beta T^\beta}{2^\beta}P_n^{\gamma,\beta}(r)}{[t-\frac{(r+1)T}{2}]^{1-\alpha}}\ \frac{T}{2}dr\nonumber\\
		&=\frac{1}{\Gamma(\alpha)}\int_{-1}^{\frac{2t}{T}-1}\frac{(r+1)^\beta P_n^{\gamma,\beta}(r)}{(\frac{2t}{T}-r-1)^{1-\alpha}}\ \frac{T^{\beta+\alpha}}{2^{\beta+\alpha}}dr\nonumber\\
		&\xlongequal[]{\text{Let $\frac{2t}{T}-1=x$}}\frac{1}{\Gamma(\alpha)}\int_{-1}^{x}\frac{(r+1)^\beta P_n^{\gamma,\beta}(r)}{(x-r)^{1-\alpha}}\ \frac{T^{\beta+\alpha}}{2^{\beta+\alpha}}dr\nonumber\\
		&={}_{-1}I_x^{\alpha}((1+x)^{\beta}P_n^{\gamma,\beta}(x))\ \frac{T^{\beta+\alpha}}{2^{\beta+\alpha}}\nonumber\\
		&=\frac{\Gamma(n+\beta+1)}{\Gamma(n+\beta+\alpha+1)}(1+x)^{\beta+\alpha}P_n^{\gamma-\alpha,\beta+\alpha}(x)\ \frac{T^{\beta+\alpha}}{2^{\beta+\alpha}}\nonumber\\
		&\xlongequal[]{\text{Let $x=\frac{2t}{T}-1$}}\frac{\Gamma(n+\beta+1)}{\Gamma(n+\beta+\alpha+1)}(1+\frac{2t}{T}-1)^{\beta+\alpha}P_n^{\gamma-\alpha,\beta+\alpha}(\frac{2t}{T}-1)\ \frac{T^{\beta+\alpha}}{2^{\beta+\alpha}}\nonumber\\
		&=\frac{\Gamma(n+\beta+1)}{\Gamma(n+\beta+\alpha+1)}t^{\beta+\alpha}Q_{n}^{\gamma-\alpha,\beta+\alpha}(t).\label{condition1}
	\end{align}
	
	\item{\bf{Fractional derivative.}}
	By performing the  operator ${}_{0}D_{t}^{\alpha}$ on both sides of the equation \eqref{condition1} and using the inverse property \eqref{inverse},  we have the fractional derivative of weighted Jacobi polynomials  with  $\gamma\in R$, $\beta-\alpha>-1$,
	\begin{eqnarray}\label{derivative}
		{}_{0}D_t^{\alpha}(t^{\beta}Q_n^{\gamma,\beta}(t))=\frac{\Gamma(n+\beta+1)}{\Gamma(n+\beta-\alpha+1)}t^{\beta-\alpha}Q_{n}^{\gamma+\alpha,\beta-\alpha}(t).
	\end{eqnarray}
	
	Similarly, we have, for $\gamma-\alpha>-1$, $\beta\in R$,
	\begin{eqnarray}\label{rderivative}
		{}_{t}D_T^{\alpha}((T-t)^{\gamma}Q_n^{\gamma,\beta}(t))=\frac{\Gamma(n+\gamma+1)}{\Gamma(n+\gamma-\alpha+1)}(T-t)^{\gamma-\alpha}Q_{n}^{\gamma-\alpha,\beta+\alpha}(t).
	\end{eqnarray}
	
\end{itemize*}

\subsection*{\bf{Fractional Sobolev space}}

Let $s>0$ be noninteger with $\nu=s-\lfloor s \rfloor>0$ being its noninteger part, where  $\lfloor s\rfloor$ is the integer part of $s$. Then the fractional Sobolev space \cite{adams1} is defined by
\begin{equation*}
	H^s(I):=\{v\in H^{\lfloor s\rfloor}(I): \int_I\int_I\frac{|D^{\lfloor s\rfloor}v(x)-D^{\lfloor s\rfloor}v(y)|^2}{|x-y|^{1+2\nu}}dxdy<\infty\}
\end{equation*}
endowed with the norm
\begin{equation*}
	\|v\|^2_{H^s}=\|v\|^2_{H^{\lfloor s\rfloor}}+\int_I\int_I\frac{|D^{\lfloor s\rfloor}v(x)-D^{\lfloor s\rfloor}v(y)|^2}{|x-y|^{1+2\nu}}dxdy.
\end{equation*}
For $s<0$ the space is defined by  $L^2$ duality.

Let ${}_0C^{\infty}(I)$ stands for the space of smooth functions with compact support in $(0,T]$, and for $s>0$, ${}_0H^s(I)$ denote the closure of ${}_0C^{\infty}(I)$ with respect to norm $\|\cdot\|_{H^s(I)}$.

\subsection*{{\bf{Weighted Sobolev spaces}}}
\begin{itemize*} 
	\item {\bf{ $L^2_{\omega^{\gamma,\beta}}(I)$.}} Let $\omega^{\gamma,\beta}(x)=(T-x)^{\gamma}x^{\beta}$, $\gamma,  \beta>-1$. Then
	\begin{equation*}
		L^2_{\omega^{\gamma,\beta}}(I):=\{v(x):\int_I\omega^{\gamma,\beta}(x)v^2(x)dx<\infty\}
	\end{equation*}
	with the inner product and norm defined by
	\begin{eqnarray*}
		(u,v)_{\omega^{\gamma,\beta}}=\int_{I}uv\omega^{\gamma,\beta}dx,\ \ \|u\|_{{\omega^{\gamma,\beta}}}=((u,u)_{\omega^{\gamma,\beta}})^{1/2}.
	\end{eqnarray*}
	When $\gamma=\beta=0$, we will drop $\omega$ from the above notations.
	\\~
	
	\item {\bf{Defined by interpolation.}} Following \cite{babu1} and \cite{guo1}, define the weighted Sobolev space, when $s$ is a nonnegative integer,
	\begin{equation*}
		H^s_{\omega^{\gamma,\beta}}(I)=\{v(x):D^kv(x)\in L^2_{\omega^{\gamma+k,\beta+k}}(I), k=0,1,\cdots,s\},
	\end{equation*}
	equipped with the norm
	\begin{equation*}
		\|v\|_{H^s_{\omega^{\gamma,\beta}}}=\left(\sum\limits_{k=0}^{s}|v|^2_{H^k_{\omega^{\gamma,\beta}}}\right)^{1/2},\ |v|_{H^k_{\omega^{\gamma,\beta}}}=\|D^kv\|_{{\omega^{\gamma+k,\beta+k}}}.
	\end{equation*}
	When $s\in R^{+}$ the space can be defined via the K-method \cite{adams1} of interpolation. For $s<0$ the space is defined by (weighted) $L^2$ duality.
	\\~

	\item {\bf{Equivalent norm in \bm{$H^s_{\omega^{\gamma,\beta}}(I)$}.}}
	For any $s$, the norm in $H^s_{\omega^{\gamma,\beta}}$ is equivalent to (see \cite{babu1})
	\begin{equation}\label{norm2}
		\|u\|_{H^s_{\omega^{\gamma,\beta}}}^{2}=\sum_{n=0}^{\infty}\left(u_{n}^{\gamma, \beta}\right)^{2}h_n^{\gamma,\beta}(1+n^2)^s
	\end{equation}
	where $\gamma, \beta>-1, u_{n}^{\gamma, \beta}=\frac{1}{h_{n}^{\gamma, \beta}} \int_{I} u(x) Q_{n}^{\gamma, \beta}(x) \omega^{\gamma, \beta} d x,$  $h_{n}^{\gamma, \beta}=\|Q_{n}^{\gamma,\beta}\|^2_{{\omega^{\gamma,\beta}}}$.
	Obviously, from \eqref{jacobi},
	\begin{eqnarray*}\label{jacobi1}
		\|Q_{n}^{\gamma,\beta}\|^2_{{\omega^{\gamma,\beta}}}=\int_0^T(T-t)^{\gamma}t^{\beta}Q_{n}^{\gamma,\beta}Q_{n}^{\gamma,\beta}dt=\left(\frac{T}{2}\right)^{\gamma+\beta+1}|\|P_n^{\gamma,\beta}\||^2.
	\end{eqnarray*}
	\\~
	
	\quad \quad Let $\mathbb{N}_0=\mathbb{N}\cup \{0\}$. In \cite{fdez1}, it is shown that for $s\geq0$, $s=\lfloor s\rfloor+\nu$, $0\leq\nu<1$ satisfied $s\neq1+\gamma$ if $\gamma\in(-1,0)$ and $s\neq1+\beta$ if $\beta\in(-1,0)$, the norm in $H^s_{\omega^{\gamma,\beta}}(I)$ is equivalent to
	\begin{eqnarray*}\begin{aligned}
			\|v\|^2_{H^s_{\omega^{\gamma,\beta}}}&=\begin{cases}\ \sum\limits_{k=0}^{\lfloor s\rfloor}\|D^kv\|^2_{{\omega^{\gamma+k,\beta+k}}},\ \text{for $s \in \mathbb{N}_0$},\\
				\ \sum\limits_{k=0}^{\lfloor s\rfloor}\|D^kv\|^2_{{\omega^{\gamma+k,\beta+k}}}+|v|^2_{H^s_{\omega^{\gamma,\beta}}},\ \text{for $s\in \mathbb{R}^{+}\backslash \mathbb{N}_0$}.
			\end{cases}\\
			|v|^2_{H^s_{\omega^{\gamma,\beta}}}&=\iint_{\Omega_{I,\zeta}}(T-x)^{\gamma+s}x^{\beta+s}\frac{|D^{\lfloor s\rfloor}v(x)-D^{\lfloor s\rfloor}v(y)|^2}{|x-y|^{1+2\nu}}dydx,
	\end{aligned}\end{eqnarray*}
	where for any $    \zeta    >1$ the domain $\Omega_{I,    \zeta    }$ is defined by
	\begin{equation*}
		\Omega_{I,    \zeta    }=\{(x,y)\in I\times I:     \zeta    ^{-1}(T-|2x-T|)<T-(2y-T)\mbox{sgn}(\frac{2x}{T}-1)<    \zeta    (T-|2x-T|)\}.
	\end{equation*}
	Taking $    \zeta    =3/2$, it follows that
	\begin{equation*}
		\Omega_{I,3/2}=\{(x,y): \frac{2}{3}x<y<\frac{3}{2}x,\ 0<x<\frac{T}{2}\}\cup\{(x,y):\frac{3}{2}x-\frac{T}{2}<y<\frac{2}{3}x+\frac{T}{3},\ \frac{T}{2}\leq x<T\}.
	\end{equation*}
	
	Whether a function $f$ lies in the space $H^s_{\omega^{\gamma,\beta}}(I)$  depends on its behavior at: (i) the left endpoint $(x=0)$, (ii) the right endpoint $(x=T)$, and (iii) away from the endpoints.
	In order to separate the consideration of the endpoint behaviors, following \cite{bern1}, we introduce the following function space $H_{(    \zeta    )}^s(J)$.
	\\~
	\item {\bf{\bm{$H_{(    \eta    )}^s(J)$}.} } Let $J=(0,\frac{3}{4}T)$, for $s\geq0$, $s=\lfloor s\rfloor+\nu$, $0\leq \nu <1$,
	\begin{equation*}
		H_{(\eta)}^s(J)=\{v(x):v(x)\ \mbox{is measurable and}\ \|v\|_{H_{(    \eta    )}^s(J)}<\infty\},
	\end{equation*}
	here the norm $\|\cdot\|_{H_{(    \eta    )}^s(J)}$ is defined by
	\begin{eqnarray*}\begin{aligned}
			\|v\|^2_{H^s_{(    \eta    )}(J)}&=\begin{cases}\ \sum\limits_{k=0}^{\lfloor s\rfloor}\|D^kv\|^2_{L^2_{(    \eta    +k)}(J)},\ \text{for $s \in \mathbb{N}_0$},\\
				\ \sum\limits_{k=0}^{\lfloor s\rfloor}\|D^kv\|^2_{L^2_{(    \eta    +k)}(J)}+|v|^2_{H^s_{(    \eta    )}(J)},\ \text{for $s\in \mathbb{R}^{+}\backslash \mathbb{N}_0$},
			\end{cases}
	\end{aligned}\end{eqnarray*}
	\begin{eqnarray*}\begin{aligned}
			\|v\|_{L^2_{(    \eta    )}(J) }^2:=\int_Jx^    \eta     v^2(x)dx,\ \
			|v|_{H^s_{(    \eta    )}(J) }^2&=\iint_{\Lambda^*}x^{    \eta    +s}\frac{|D^{\lfloor s\rfloor}v(x)-D^{\lfloor s\rfloor}v(y)|^2}{|x-y|^{1+2\nu}}dydx
	\end{aligned}\end{eqnarray*}
	where
	\begin{equation*}\begin{aligned}
			\Lambda^{*}&=\{(x,y): \frac{2}{3}x<y<\frac{3}{2}x,\ 0<x<\frac{T}{2}\}\cup\{(x,y):\frac{3}{2}x-\frac{T}{2}<y<\frac{2}{3}x+\frac{T}{3},\ \frac{T}{2}\leq x<\frac{3}{4}T\}.
			%:&=\Lambda\cup\Lambda_1.
	\end{aligned}\end{equation*}
	By the definition of $H^s_{\omega^{\gamma,\beta}}(I)$ and $ H^s_{(\zeta)}(J)$, the following result can be readily obtained.
	\begin{lem}\label{lem:2}
		A function $v\in H^s_{\omega^{\gamma,\beta}}(I)$ if and only if $v\in H^s_{(\beta)}(J)$ and $\widehat{v}\in H^s_{(\gamma)}(J)$, where  $\widehat{v}(x):=v(T-x)$.
	\end{lem}
	\begin{proof} For $s\in N_0$, $v\in H^s_{\omega^{\gamma,\beta}}(I)$ is equivalent to $\|v\|_{H^s_{\omega^{\gamma,\beta}}}<+\infty$, i.e.
		\begin{align}
			\|v\|_{H^s_{\omega^{\gamma,\beta}}}^2= \sum\limits_{k=0}^{ s}\|D^kv\|^2_{{\omega^{\gamma+k,\beta+k}}}= \sum\limits_{k=0}^{ s}\int_0^{T}\omega^{\gamma+k,\beta+k}(D^kv)^2dx<\infty.\nonumber
		\end{align}
		For  $\widehat{v}(x)=v(T-x)$, we can write
		\begin{align}
			\|D^k\widehat{v}\|^2_{L^2_{(\gamma+k)}(J)}=\int_{0}^{\frac{3}{4}T}t^{\gamma+k}(D^k v(T-t))^2dt=\int_{\frac{1}{4}T}^{T}(T-x)^{\gamma+k}(D^kv)^2dx.\nonumber
		\end{align}
		If $v\in H^s_{\omega^{\gamma,\beta}}(I)$, noting that $\omega^{\gamma+k,0}\in C^{\infty}[0,\frac{3}{4}T]$, $\omega^{0,\beta+k} \in C^{\infty}[\frac{1}{4}T,T]$ and the definition of $\|\cdot\|_{L_{(    \eta    )}^2 (J)}$, we can obtain
		\begin{align*}
			\|D^kv\|^2_{{\omega^{\gamma+k,\beta+k}}}&=\int_0^{T}\omega^{\gamma+k,\beta+k}(D^kv)^2dx\geq \int_{\frac{1}{4}T}^T\omega^{\gamma+k,\beta+k}(D^kv)^2dx \nonumber\\
			&\geq C\int_{\frac{1}{4}T}^{T}(T-x)^{\gamma+k}(D^kv)^2dx= C\|D^k\widehat{v}\|^2_{ L^2_{(\gamma+k)}(J)},\label{equivalent2}
		\end{align*}
		and similarly
		\begin{equation*}
			\|D^kv\|^2_{{\omega^{\gamma+k,\beta+k}}}\geq\int_0^{\frac{3}{4}T}\omega^{\gamma+k,\beta+k}(D^kv)^2dx\geq C\int_0^{\frac{3}{4}T}x^{\beta+k}(D^kv)^2dx= C\|D^kv\|^2_{ L^2_{(\beta+k)}(J)}.
		\end{equation*}
		Thus $v\in H^s_{(\beta)}(J)$ and $\widehat{v}\in H^s_{(\gamma)}(J)$ by the definition of  $ H^s_{(\eta)}(J)$. On the other hand,
		\begin{align*}
			\|D^kv\|^2_{{\omega^{\gamma+k,\beta+k}}}&\leq\int_0^{\frac{3}{4}T}\omega^{\gamma+k,\beta+k}(D^kv)^2dx+\int_{\frac{1}{4}T}^{T}\omega^{\gamma+k,\beta+k}(D^kv)^2dx\nonumber\\
			&\leq C(\|D^kv\|^2_{L^2_{(\beta+k)}(J)}+\|D^k\widehat{v}\|^2_{ L^2_{(\gamma+k)}(J)})\label{equivalent3}.
		\end{align*}
		The conclusion is proved when $s$ are nonnegative integers.
		For $s\in R^{+}$, it can be obtained by the space interpolation theory \cite{adams1}.
	\end{proof}
\end{itemize*}

\section{Regularity}\label{section3}
In this section, we investigate the regularity of the solution to \eqref{FIVP} in weighted Sobolev space. Throughout the paper, $C$ and $c$ denote  generic constants independent of the truncation number $N$.

The weak formulation of the problem \eqref{FIVP} is: give $f\in H^{-\frac{\alpha}{2}}(I)$, to find $u\in {}_0H^{\frac{\alpha}{2}}(I)$ such that
\begin{eqnarray}\label{weak}
	a(u,v):=\left({}_{0}D_t^{\frac{\alpha}{2}}u,{}_{t}D_T^{\frac{\alpha}{2}}v \right)+\lambda(u,v)=\langle f,v\rangle,\ \forall\, v\in {}_0H^{\frac{\alpha}{2}}(I).
\end{eqnarray}

The well-posedness of the problem \eqref{weak} can be established  refer to \cite{ervin1}. As no new techniques are used here, we omit the proof and give the result as follows.
\begin{thm}\label{T1}
	For all $0<\alpha<1$ and $f\in H^{-\frac{\alpha}{2}}(I)$, the problem \eqref{weak} exists a unique solution $u\in {}_0H^{\frac{\alpha}{2}}(I)$ such that $$\|u\|_{H^{\frac{\alpha}{2}} }\leq C\|f\|_{H^{-\frac{\alpha}{2}} }.$$
\end{thm}

The regularity in weighted Sobolev space has been proved for the solution of initial value problem ${}_{-1}D_x^{\alpha}u=f(x),\ x\in (-1,1)$ in \cite{zhang1}. By using variable substitution, we can obtain the regularity results for the problem \eqref{FIVP} with $\lambda=0$ directly.

\begin{thm}\label{T2}
	For the problem \eqref{FIVP} with $\lambda=0$, assume that $\gamma>\alpha-1$. If $\omega^{0,\alpha-\beta}f\in H^r_{\omega^{\gamma,\beta-\alpha}}(I)$ $(r\geq0)$ with $\beta>1-\alpha$ or $\omega^{0,\alpha-\beta}f\in H^r_{\omega^{\gamma,\beta-\alpha}}(I)\cap C\left(\bar{I}\right)$ $(r\geq0)$, then $\omega^{0,-\beta}u\in H^{r+\alpha}_{\omega^{\gamma-\alpha,\beta}}(I)$.
\end{thm}

\begin{lem}[\cite{bern2}]\label{Bern2} For any $r$ such that $1<r<+\infty$ and for any real number $\beta<-1$ the following inequalities hold:  $\forall\phi \in C_0^{\infty}(0,1]$,
	\begin{eqnarray}\label{Hardy}\begin{aligned}
			\int_0^1|\phi(x)|^rx^{\beta}dx\leq \left(\frac{r}{|\beta+1|}\right)^r\int_0^1|\phi'(x)|^rx^{\beta+r}dx.
	\end{aligned}\end{eqnarray}
\end{lem}

In \cite{ervin2}, the authors give the restrictions of $q$ and $\sigma$ to ensure $x^p\phi(x)\in H^{q}_{(\sigma)}(J)$ for $\phi(x)\in H^{s}_{(\mu)}(J)$ with $J=(0,\frac{3}{4})$. Here for $J=(0,\frac{3T}{4})$ the following results can be obtained similarly.

\begin{lem}\label{ervin}
	Let $n\leq s<n+1$, $n\in \mathbb{N}_0$, $p\geq0$,  $\mu>-1$, and $\psi\in H^{s}_{(\mu)}(J)$. If
	\begin{equation*}
		0\leq q\leq s,\ \sigma+2p\geq\mu,\ \sigma+2p-q>-1,\ \sigma+2p+q\geq\mu+s,
	\end{equation*}
	then $x^p\psi\in H^{q}_{(\sigma)}(J)$. Moreover, there exists a positive constant $C$ independent of $\psi$, such that
	\begin{equation}\label{condition}
		\|x^p\psi\|_{H^{q}_{(\sigma)}(J)}\leq C\|\psi\|_{H^{s}_{(\mu)}(J)}.
	\end{equation}
\end{lem}

\begin{rem}
	Except for the different regions considered, the only difference between Lemma \ref{ervin} and the Theorem 6.4 in \cite{ervin2} is the use of an additional condition $p\geq n$ in the latter. It is unreasonable as $p$ is generally a small number. Hence,  it is suffice to state that the condition $p\geq n$ can be removed.
	
	Actually, for $n\leq s<n+1$, the condition $p\geq n$ is presented in the Theorem 6.4 in \cite{ervin2} to bound the integral by using Hardy's inequality \eqref{Hardy}:
	\begin{align*}
		&\int_{2/3}^{3/2}|1-z|^{-1-2(q-n)}|z^{p-j}-1|^2dz\\
		&=\int_{2/3}^{1}(1-z)^{-1-2(q-n)}(z^{p-j}-1)^2dz+\int_1^{3/2}(z-1)^{-1-2(q-n)}(z^{p-j}-1)^2dz\\
		&\leq C\int_{2/3}^{1}(1-z)^{-1-2(q-n)+2}(z^{p-j-1})^2dz+\int_1^{3/2}(z-1)^{-1-2(q-n)+2}(z^{p-j-1})^2dz\\
		&\leq C,\ \mbox{provided} \ q<n+1.
	\end{align*}
	However, from Lemma \ref{Bern2} we find that there is no need to require $p\geq j$, $j=0,1,\cdots, n$, when using \eqref{Hardy}.
\end{rem}
\begin{lem}[\cite{ervin2}]\label{Ervinadd}
	Let $s\geq0$, $\phi\in H_{(\gamma)}^s(J)$, and $g\in C^{\lceil s\rceil}(J)$. Then
	\begin{equation}\label{6.2}
		\|g\phi\|_{H_{(\gamma)}^s(J)}\leq C\|g\|_{C^{\lceil s\rceil}(J)}\|\phi\|_{H_{(\gamma)}^s(J)},
	\end{equation}
	where $\lceil s\rceil$ is the smallest integer greater than $s$.
\end{lem}

\begin{lem}\label{lemmaadd}
	If $\phi(x)\in H^{s}_{\omega^{0,\alpha}}(I)$ with $0\leq s< 3\alpha+1$, then $\omega^{0,\alpha}\phi(x)  \in H^{\min(s,2\alpha+1-\epsilon)}_{\omega^{\alpha,0}}(I)$ with arbitrarily small $\epsilon>0$.
\end{lem}
\begin{proof}
	By Lemma \ref{lem:2}, we have that
	\begin{equation}\label{eqadd:1}
		\phi(x)\in H^{s}_{\omega^{0,\alpha}}(I)\Longleftrightarrow	\phi(x)\in H^{s}_{(\alpha)}(J) \; \text{and}\; \widehat{\phi}(x)\in H^{s}_{(0)}(J),
	\end{equation}
	where  $\widehat{\phi}(x):=\phi(T-x)$.
	
	Let $\Psi(x):=\omega^{0,\alpha}(x)\phi(x)= x^{\alpha}\phi(x)$ and denote $\widehat{\Psi}(x):=\Psi(T-x)=\left(T-x\right)^{\alpha}\widehat{\phi}(x)$.
	
	When $\phi(x)\in H^{s}_{\omega^{0,\alpha}}(I)$,  by using \eqref{eqadd:1} and Lemma \ref{ervin} with $\mu=\alpha$, $p=\alpha$, choosing $\sigma=0$, we have
	$\Psi(x)\in H^{\min(s,2\alpha+1-\epsilon)}_{(0)}(J)$.
	Moreover, by the definition of space $H^{s}_{(\zeta)}(J)$, we have
	\begin{equation*}
		\|\widehat{\phi}\|_{H^{s}_{(\alpha)}(J)}\leq C\|\widehat{\phi}\|_{H^{s}_{(0)}(J)},
	\end{equation*}
	thus $\widehat{\phi}(x)\in H^{s}_{(\alpha)}(J)$.
	Using Lemma \ref{Ervinadd}, it follows that
	\begin{equation*}
		\|\widehat{\Psi}\|_{H^{s}_{(\alpha)}(J)}\leq \|\omega^{\alpha,0}\|_{C^{\lceil s\rceil}(J)} \|\widehat{\phi}\|_{H^{s}_{(\alpha)}(J)},
	\end{equation*}
	that is,  $\widehat{\Psi}(x)\in H^{s}_{(\alpha)}(J)$.
	
	Based on above analysis,  using Lemma \ref{lem:2} again, we obtain that $\Psi(x)\in H^{\min(s,2\alpha+1-\epsilon)}_{\omega^{\alpha,0}}(I)$.
\end{proof}

\begin{thm}\label{theorem3.6}
	For the problem \eqref{FIVP} with  $\lambda>0$, if $f\in H^r_{\omega^{\alpha,0}}(I)\cap H^{-\frac{\alpha}{2}}(I)$ with $r\geq0$, then we have $\omega^{0,-\alpha}u \in H^{\alpha+\min(2\alpha+1-\epsilon,r)}_{\omega^{0,\alpha}}(I)$.
\end{thm}
\begin{proof}
	For $f\in H^r_{\omega^{\alpha,0}}(I)\cap H^{-\frac{\alpha}{2}}(I)$ with $r\geq0$, by Theorem \ref{T1}, there exists a unique solution $u\in H^{\frac{\alpha}{2}}(I)$. Note that $H^{\frac{\alpha}{2}}(I)\subset H^{\frac{\alpha}{2}}_{\omega^{\alpha,0}}(I)$. In fact,
	\begin{eqnarray*}\begin{aligned}
			\|u\|_{H^{\frac{\alpha}{2}}_{\omega^{\alpha,0}} }&=\|u\|^2_{{\omega^{\alpha,0}} }+\iint_{\Omega_{I,3/2}}(T-x)^{\frac{3\alpha}{2}}x^{\frac{\alpha}{2}}\frac{|u(x)-u(y)|^2}{|x-y|^{1+\alpha}}dydx\\
			&\leq C\Big[\|u\|^2_{L^{2} }+\int_I\int_I\frac{|u(x)-u(y)|^2}{|x-y|^{1+\alpha}}dydx\Big]=C\|u\|_{H^{\frac{\alpha}{2}} }.
	\end{aligned}\end{eqnarray*}
	Hence, by $u\in H^{\frac{\alpha}{2}}_{\omega^{\alpha,0}}(I)$, we have
	${}_{0}D_t^{\alpha}u=f-\lambda u:=\hat{f}\in H^{\min(\frac{\alpha}{2},r)}_{\omega^{\alpha,0}}(I)$. By using Theorem \ref{T2} with $\gamma=\beta=\alpha$, we have $\omega^{0,-\alpha}u\in H^{\alpha+\min(\frac{\alpha}{2},r)}_{\omega^{0,\alpha}}(I)$.
	
	If $r\geq \frac{\alpha}{2}$, then $\omega^{0,-\alpha}u\in H^{\frac{3\alpha}{2}}_{\omega^{0,\alpha}}(I)$,
	by using Lemma \ref{lemmaadd} we have $u\in H^{\min(\frac{3\alpha}{2},2\alpha+1-\epsilon)}_{\omega^{\alpha,0}}(I)$. Then $\hat{f}\in H^{\min(\frac{3\alpha}{2},2\alpha+1-\epsilon,r)}_{\omega^{\alpha,0}}(I)$,  and $\omega^{0,-\alpha}u\in H^{\alpha+\min(\frac{3\alpha}{2},2\alpha+1-\epsilon,r)}_{\omega^{0,\alpha}}(I)$.
	
	Similarly, if $r\geq \frac{3\alpha}{2}$, we can follow the argument to lift the regularity. Suppose that after repeating the lifting procedure $k$ times, where $k$ is the least integer such that $(k+1/2)\alpha\geq 2\alpha+1-\epsilon,$ we have $$\omega^{0,-\alpha}u\in H^{\alpha+\min((k+1/2)\alpha,2\alpha+1-\epsilon,r)}_{\omega^{0,\alpha}}(I)=H^{\alpha+\min(2\alpha+1-\epsilon,r)}_{\omega^{0,\alpha}}(I).$$
\end{proof}

\begin{lem}[\cite{ervin2}]\label{Lem}
	Let $v(x)=x^{\mu}$. Then $v\in H^{s}_{\omega^{\gamma,\beta}}(I)$ for $s<2\mu+\beta+1$.
\end{lem}

\begin{rem}\label{re}
	It is well known that the FIVP \eqref{FIVP} is equivalent to the weakly singular Volterra integral equation
	\begin{equation}\label{req1}
		u(t)=g(t)-{}_{0}I_t^{\alpha}(\lambda u(t)),\  t\in [0,T].
	\end{equation}
	where $g(t)=({}_{0}I_t^{\alpha}f)(t)$.
	%Following \cite{hu}, if $f\in C^1[0,T]$,  we write $f(t)=f(0)+\int_{0}^tf^{'}(s)ds$, then
	%\begin{equation}\label{g}
	%g(t)=\frac{f(0)}{\Gamma(\alpha+1)}t^{\alpha}+\frac{1}{\Gamma(\alpha+1)}\int_0^t f'(s)(t-s)^{\alpha}ds:=\frac{f(0)}{\Gamma(\alpha+1)}t^{\alpha}+\psi(t),
	%\end{equation}
	% where $\psi\in C^1[0,T]$.
	Following \cite{bru} (Theorems 6.1.2 and 6.1.6) and \cite{hu}, if $f\in C^1[0,T]$, the solution of \eqref{req1} is
	%\begin{equation}\label{solu}
	%u(t)=g(t)+\int_0^tR_{1-\alpha}g(s)ds, \ t\in [0,T],
	%\end{equation}
	%where $R_{1-\alpha}(t, s)=(t-s)^{\alpha-1} \sum_{n=1}^{\infty}(t-s)^{(n-1) \alpha} \Phi^{\alpha}_{n},$ with $\Phi^{\alpha}_{n}=\frac{\lambda\Gamma(n\alpha-\alpha)}{\Gamma(n\alpha)}\Phi^{\alpha}_{n-1}$ ($n\geq2$), $\Phi^{\alpha}_{1}=\frac{\lambda}{\Gamma(\alpha)}$.  Substituting \eqref{g} into \eqref{solu} and imitating the proof of Theorem 6.1.6 (ii) in \cite{bru}, we obtain
	\begin{equation*}
		u(t)=\frac{f(0)}{\Gamma(\alpha+1)}t^{\alpha}%+\frac{f(0)}{\Gamma(\alpha+1)}\int_0^tR_{1-\alpha}\cdot s^{\alpha}ds+
		+\sum\limits_{k=1}^{\infty}b_k\cdot t^{k\alpha+1},
	\end{equation*}
	where coefficients $b_k$ are some constants. Thus when $f(0)=0$, by Lemma \ref{Lem}, the regularity index of solution $u$ is at least $s=2(\alpha+1)+1-\epsilon$ in $H^{s}_{\omega^{\alpha,0}}(I)$. Note that ${}_{0}D_t^{\alpha}u=f-\lambda u:=\hat{f}\in H^{\min(2\alpha+3-\epsilon,r)}_{\omega^{\alpha,0}}(I)$. By Theorem \ref{T2} we have that $\omega^{0,-\alpha}u \in H^{\alpha+\min(2\alpha+3-\epsilon,r)}_{\omega^{0,\alpha}}(I)$, which  will be verified in Example \ref{ex3}.
\end{rem}

\section{ Petrov-Galerkin formulation }\label{section4}
In this section, we introduce a Petrov-Galerkin formulation and study its well-posedness and regularity. The idea has been well used in \cite{hao2}.   We consider the  Petrov-Galerkin ultra-weak formulation of the FIVPs \eqref{FIVP}: Given $f\in H^{-\alpha}_{\omega^{\alpha,0}}(I)\ \cap \ H^{-\frac{\alpha}{2}}(I)$, to find $u\in L^2_{\omega^{0,-\alpha}}(I)$ such that
\begin{equation}\label{uweak}
	b(u,v):=(u, {}_{t}D_T^{\alpha}(\omega^{\alpha,0}v))+\lambda(u,v)_{\omega^{\alpha,0}}=\langle f,v\rangle_{\omega^{\alpha,0}},\ \forall v\in H^{\alpha}_{\omega^{\alpha,0}}(I),
\end{equation}
where $\langle\cdot,\cdot\rangle_{\omega^{\gamma,\beta}}$ denotes the $L^2_{\omega^{\gamma,\beta}}$ duality pair of $H^{-\mu}_{\omega^{\gamma,\beta}}(I)$ and $H^{\mu}_{\omega^{\gamma,\beta}}(I)$, $\gamma,\beta>-1$ and $\mu\geq0$.

To establish the well-posedness of this problem, we need to consider the adjoint problem of original problem \eqref{FIVP}:
\begin{eqnarray}\label{aFIVP}\begin{aligned}
		{}_{t}D_T^{\alpha}z+\lambda z&=g(t),\ t\in I=(0,T),\\
		z(T)&=0.
\end{aligned}\end{eqnarray}
Its  Petrov-Galerkin weak formulation is: Given $g\in L^2_{\omega^{0,\alpha}}(I)\ \cap \ H^{-\frac{\alpha}{2}}(I)$, to find $ \hat{z} \in H^{\alpha}_{\omega^{\alpha,0}}(I) $ such that
\begin{equation}\label{adjoint}
	\widetilde{a}(\hat{z},w):=( {}_{t}D_T^{\alpha}(\omega^{\alpha,0}\hat{z}),w)_{\omega^{0,\alpha}}+\lambda(\omega^{\alpha,0}\hat{z},w)_{\omega^{0,\alpha}}=\langle g,w\rangle_{\omega^{0,\alpha}},\ \forall w\in L^2_{\omega^{0,\alpha}}(I).
\end{equation}

The well-posedness of problem \eqref{adjoint} can be established similarly as that in Theorems \ref{T1} and \ref{theorem3.6}.%An analogous argument %as used to establish the well-posedness of \eqref{FIVP}
%used in Theorems \ref{T1} and \ref{theorem3.6} can be applied to \eqref{adjoint} to obtain its well-posedness as follows.
\begin{lem}\label{awell}
	For $g\in L^2_{\omega^{0,\alpha}}(I)\ \cap \ H^{-\frac{\alpha}{2}}(I)$, there exists a unique solution $\hat{z}\in  H^{\alpha}_{\omega^{\alpha,0}}(I)$ to \eqref{adjoint} satisfying
	\begin{equation}\label{awellp}
		\|\hat{z}\|_{H^{\alpha}_{\omega^{\alpha,0}}}\leq C\|g\|_{{\omega^{0,\alpha}}}.
	\end{equation}
	% where $\hat{z}(x)=\omega^{-\alpha,0}z(x)$.
\end{lem}

%The well-posedness of this problem can be sated as follows.
% and Lemma \ref{awell}.
\begin{thm}\label{rdata}
	For $f\in H^{-\alpha}_{\omega^{\alpha,0}}(I)\ \cap \ H^{-\frac{\alpha}{2}}(I)$,  there exists a unique solution $u\in L^2_{\omega^{0,-\alpha}}(I)$ to \eqref{uweak} such that
	\begin{equation}\label{uwell}
		\|u\|_{{\omega^{0,-\alpha}}}\leq C\|f\|_{H^{-\alpha}_{\omega^{\alpha,0}}}.
	\end{equation}
\end{thm}
\begin{proof}  The well-posedness of problem \eqref{uweak} is guaranteed by the well-known  Babu\v{s}ka-Aziz theorem \cite{aziz1}.
	
	To establish the continuity of bilinear form $b(\cdot,\cdot)$ on $L^2_{\omega^{0,-\alpha}}(I) \times H^{\alpha}_{\omega^{\alpha,0}}(I)$,
	we  write
	\begin{align}\label{j1}
		v=\sum\limits_{n=0}^{\infty}v_n^{\alpha,0}Q_n^{\alpha,0}(x),\  \forall  v\in H^{\alpha}_{\omega^{\alpha,0}}(I).
	\end{align}
	Let $\lambda_n^{\alpha}=\frac{\Gamma(n+\alpha+1)}{\Gamma(n+1)}$. Note that by Stirling's formula, $\lambda_n^{\alpha}\approx n^{\alpha}$, i.e. there exists $c_1, c_2>0$ such that $c_1n^{\alpha} \leq \lambda_n^{\alpha} \leq c_2n^{\alpha}$.
	By the Cauchy-Schwarz inequality, \eqref{j1}, \eqref{rderivative} and the fact that $h_n^{0,\alpha}=h_n^{\alpha,0}$, %and the property of Jacobi polynomials,
	we have
	\begin{align}\label{bic}
		(u,{}_{t}D_T^{\alpha}(\omega^{\alpha,0}v))&\leq \|u\|_{{\omega^{0,-\alpha}}} \|{}_{t}D_T^{\alpha}(\omega^{\alpha,0}v)\|_{{\omega^{0,\alpha}}}\nonumber\\
		&=\|u\|_{{\omega^{0,-\alpha}}}\Big(\sum\limits_{n=0}^{\infty}(v_n^{\alpha,0})^2(\lambda_n^{\alpha})^2h_n^{0,\alpha}\Big)^{1/2}\nonumber\\
		&\leq C\|u\|_{{\omega^{0,-\alpha}}} \|v\|_{H^{\alpha}_{\omega^{\alpha,0}}}\nonumber
	\end{align}
	and
	\begin{equation*}
		(u,v)_{\omega^{\alpha,0}}\leq \|u\|_{{\omega^{\alpha,0}}}\|v\|_{{\omega^{\alpha,0}}} \leq C\|u\|_{{\omega^{0,-\alpha}}}\|v\|_{H^{\alpha}_{\omega^{\alpha,0}}}.
	\end{equation*}
	Thus for $u\in L^2_{\omega^{0,-\alpha}}(I)$ and $v\in H^{\alpha}_{\omega^{\alpha,0}}(I)$,
	\begin{equation}\label{continuity}
		|b(u,v)|\leq C\|u\|_{{\omega^{0,-\alpha}}}\|v\|_{H^{\alpha}_{\omega^{\alpha,0}}}.
	\end{equation}
	Moreover, for any $w\in L^2_{\omega^{0,\alpha}}(I)$, $\hat{w}:=\omega^{0,\alpha}w\in L^2_{\omega^{0,-\alpha}}(I)$, the weak formulation \eqref{adjoint} can be rewritten as: Given $g\in L^2_{\omega^{0,\alpha}}(I)\ \cap \ H^{-\frac{\alpha}{2}}(I)$, to find $ \hat{z} \in H^{\alpha}_{\omega^{\alpha,0}}(I) $ such that
	\begin{equation*}\label{rewrite}
		\widetilde{a}(\hat{z},\omega^{0,-\alpha}\hat{w})=( {}_{t}D_T^{\alpha}(\omega^{\alpha,0}\hat{z}),\hat{w})+\lambda(\omega^{\alpha,0}\hat{z},\hat{w})=\langle g,\hat{w}\rangle,\ \forall \hat{w}\in L^2_{\omega^{0,-\alpha}}(I).
	\end{equation*}
	Then for any $u\in L^2_{\omega^{0,-\alpha}}(I)$, $\hat{u}:=\omega^{0,-\alpha}u\in L^2_{\omega^{0,\alpha}}(I)$ and thus $\hat{u}\in H^{-\frac{\alpha}{2}}(I)$, by Lemma \ref{awell}, there exists a unique $\hat{v}\in H^{\alpha}_{\omega^{\alpha,0}}(I)$ such that
	\begin{equation}\label{atilde}
		\widetilde{a}(\hat{v},\omega^{0,-\alpha}\hat{w})=(\hat{u},\hat{w})=(u,\hat{w})_{\omega^{0,-\alpha}}, \ \forall \hat{w}\in L^2_{\omega^{0,-\alpha}}(I),
	\end{equation}
	with
	\begin{equation}\label{thus}
		\|\hat{v}\|_{H^{\alpha}_{\omega^{\alpha,0}}}\leq C\|\omega^{0,-\alpha}u\|_{{\omega^{0,\alpha}}}=C\|u\|_{{\omega^{0,-\alpha}}}.
	\end{equation}
	Substituting $\hat{w}$ by $u$ in \eqref{atilde}, we have
	$$b(u,\hat{v})=(u, {}_{t}D_T^{\alpha}(\omega^{\alpha,0}\hat{v}))+\lambda(u,\hat{v})_{\omega^{\alpha,0}}=\widetilde{a}(\hat{v},\omega^{0,-\alpha}u)=\|u\|^2_{{\omega^{0,-\alpha}}}.$$
	
	According to above derivation, for any $u \in L^2_{\omega^{0,-\alpha}}(I)$, there exists a unique $\hat{v}\in H^{\alpha}_{\omega^{\alpha,0}}(I)$ such that $b(u,\hat{v})=\|u\|^2_{{\omega^{0,-\alpha}}}$.
	Combining it with \eqref{thus} we obtain
	\begin{equation}\label{infsup}
		\sup\limits_{0\not=v\in H^{\alpha}_{\omega^{\alpha,0}}}\frac{b(u,v)}{\|v\|_{H^{\alpha}_{\omega^{\alpha,0}}}}\geq \frac{\|u\|^2_{{\omega^{0,-\alpha}}}}{\|\hat{v}\|_{H^{\alpha}_{\omega^{\alpha,0}}}}\geq \frac{1}{C}\|u\|_{{\omega^{0,-\alpha}}},\ \forall\ 0 \not=u \in L^2_{\omega^{0,-\alpha}}(I).
	\end{equation}
	
	For any $0 \not=v\in H^{\alpha}_{\omega^{\alpha,0}}(I)$, taking $\omega^{0,-\alpha}u={}_{t}D_T^{\alpha}(\omega^{\alpha,0}v)+\lambda\omega^{\alpha,0}v$ in \eqref{uweak}, we have
	\begin{equation}\label{tinfsup}
		\sup\limits_{0\not=u\in L^2_{\omega^{0,-\alpha}}}b(u,v)\geq \|{}_{t}D_T^{\alpha}(\omega^{\alpha,0}v)+\lambda\omega^{\alpha,0}v\|^2_{{\omega^{0,\alpha}}}>0,\ \forall \ 0 \not=v\in H^{\alpha}_{\omega^{\alpha,0}}(I).
	\end{equation}
	
	Combining \eqref{continuity}, \eqref{infsup} and \eqref{tinfsup}, we obtain that there exists a unique solution $u\in L^2_{\omega^{0,-\alpha}}(I)$ such that \eqref{uwell} holds. %by the Babu\v{s}ka-Aziz theorem.
\end{proof}

The above theorem leads to the following regularity result with more general data.
\begin{thm}\label{Tadd1}
	For the problem \eqref{uweak} with $\lambda=0$, if $f\in H^r_{\omega^{\alpha,0}}(I)\cap H^{-\frac{\alpha}{2}}(I)$ and $r\geq-\alpha$,   then $\omega^{0,-\alpha}u \in H^{r+\alpha}_{\omega^{0,\alpha}}(I)$.
\end{thm}
\begin{proof}
	For $f\in H^{r}_{\omega^{\alpha,0}}(I)\cap H^{-\frac{\alpha}{2}}(I)$ with $r\geq -\alpha$, by Theorem \ref{rdata}, we have $u\in L^2_{\omega^{0,-\alpha}}(I)$, then $\omega^{0,-\alpha}u\in L^2_{\omega^{0,\alpha}}(I)$. It is legitimate to write
	\begin{equation*}\label{uf}
		u=\omega^{0,\alpha}\sum\limits_{n=0}^{\infty}u_nQ_{n}^{0, \alpha}(x), \quad f=\sum_{n=0}^{\infty} f_{n} Q_{n}^{\alpha, 0}(x).
	\end{equation*}
	By \eqref{derivative} and the equation ${}_0D^{\alpha}_t u=f$, we have
	\begin{equation*}
		{}_{0}D_t^{\alpha}u=\sum\limits_{n=0}^{\infty}u_n \lambda_n^{\alpha}Q_{n}^{\alpha, 0}(x)=\sum_{n=0}^{\infty} f_{n} Q_{n}^{\alpha, 0}(x),
	\end{equation*}

	thus $u_n=f_n/\lambda_n^{\alpha}$.  By the  norm  \eqref{norm2} and using  $\lambda_{n}^{\alpha} \approx n^{\alpha}$, it follows that
	\begin{align}
		\left\|\omega^{0,-\alpha} u\right\|_{H_{\omega^{0,\alpha}}^{\alpha+r}}^{2} & = \sum_{n=0}^{\infty}\left(u_{n}\right)^{2} h_{n}^{0,\alpha}\left(1+n^{2}\right)^{\alpha+r}=\sum_{n=0}^{\infty}\left(f_{n}/ \lambda_{n}^{\alpha}\right)^{2} h_{n}^{\alpha,0}\left(1+n^{2}\right)^{\alpha+r} \nonumber\\
		& \leq C \sum_{n=0}^{\infty}\left(f_{n}\right)^{2} h_{n}^{\alpha,0}\left(1+n^{2}\right)^{r} = C\|f\|_{H^{r}_{\omega^{\alpha,0}}}^{2}.\label{sp}
	\end{align}
	This completes the proof.
	
\end{proof}

For $\lambda>0,$ the following regularity result can be obtained by the same argument based on the bootstrapping technique used in Theorem \ref{theorem3.6}.
\begin{thm}\label{hregularity}
	For the problem \eqref{uweak} with $\lambda> 0$, if $f\in H^{r}_{\omega^{\alpha,0}}(I)\cap H^{-\frac{\alpha}{2}}(I)$ with $r\geq -\alpha$, then for any $\epsilon>0$, we have %the following stability estimate
	\begin{equation}\label{hr}
		\|\omega^{0,-\alpha}u\|_{H^{\alpha+\min(2\alpha+1-\epsilon,r)}_{\omega^{0,\alpha}}}\leq C\|f\|_{H^{r}_{\omega^{\alpha,0}}}.
	\end{equation}
\end{thm}

\section{Spectral Petrov-Galerkin method}\label{section5}
Now, we are in the position to consider a spectral Petrov-Galerkin method  for the FIVPs \eqref{FIVP} and present its stability  and error estimate.

We define the finite-dimensional spaces,
\begin{align}
	U_N=\{u|u=t^\alpha v, v\in V_N(I)\},\ V_N=P_N(I), \label{space}
\end{align}
where $P_N(I)$ is the set of all algebraic polynomials of degree at most $N$ in $I$.
The spectral Petrov-Galerkin method is: Given $f\in H^{r}_{\omega^{\alpha,0}}(I)\ \cap \ H^{-\frac{\alpha}{2}}(I)$ with $r\geq-\alpha$, to find $u_N\in U_N$ such that
\begin{eqnarray}\label{scheme1}\begin{aligned}
		b(u_N,v_N)=(u_N,{}_{t}D_T^{\alpha}(\omega^{\alpha,0}v_N))+\lambda(u_N,v_N)_{\omega^{\alpha,0}}=\langle f,v_N\rangle_{\omega^{\alpha,0}},\ \forall v_N\in V_N.
\end{aligned}\end{eqnarray}

Let $\pi_N^{\gamma,\beta}: L^2_{\omega^{\gamma,\beta}}(I)\rightarrow P_N(I) $, $\gamma,\ \beta>-1$, is the $L^2_{\omega^{\gamma,\beta}}(I)$-orthogonal projection
\begin{eqnarray}\label{projection}
	(\pi_N^{\gamma,\beta}u-u,v)_{\omega^{\gamma,\beta}}=0, \ \forall v\in P_N(I),
\end{eqnarray}
%The operator $\pi_N^{\gamma,\beta}$
which can also be expressed by
\begin{eqnarray}\label{projection1}
	\pi_N^{\gamma,\beta}u(x)=\sum\limits_{n=0}^{N}\hat{u}_nQ_n^{\gamma,\beta},\ \hat{u}_n=\frac{(u,Q_n^{\gamma,\beta})_{\omega^{\gamma,\beta}}}{\|Q_n^{\gamma,\beta}\|^2_{{\omega^{\gamma,\beta}}}}.
\end{eqnarray}

To  implement the scheme \eqref{scheme1}, we need to take $u_N$, $v_N$ of the form
\begin{eqnarray}\label{form}
	u_N=t^{\alpha}\sum\limits_{n=0}^{N}u_nQ_n^{0,\alpha},\ \ \ v_N=Q_j^{\alpha,0},\ \ j=0,1,\cdots,N.
\end{eqnarray}

Before presenting our theoretical analysis, we need the following projection error estimate.

\begin{lem}[ \cite{guo1}] \label{lemma1}
	For any $v\in H_{\omega^{\gamma,\beta}}^r(I)$ and for all $0\leq r_1 \leq r$,
	\begin{eqnarray}\label{perror}
		\|\pi_N^{\gamma,\beta}v-v\|_{H^{r_1}_{\omega^{\gamma,\beta}}}\leq C(N(N+\gamma+\beta))^{\frac{r_1-r}{2}}|v|_{H^r_{\omega^{\gamma,\beta}}},
	\end{eqnarray}
	where $C$ is a generic positive constant independent of any function $v$, $N$, $\gamma$, $\beta$.
\end{lem}

\begin{lem}[\cite{hao1}]\label{lemma2}
	For any $\omega^{-\gamma,-\beta}v\in H_{\omega^{\gamma,\beta}}^r(I)$ with $0\leq r \leq N$,
	\begin{eqnarray}\label{perror}
		\|\omega^{\gamma,\beta}\pi_N^{\gamma,\beta}(\omega^{-\gamma,-\beta}v)-v\|_{{\omega^{-\gamma,-\beta}}}\leq CN^{-r}|\omega^{-\gamma,-\beta}v|_{H^r_{\omega^{\gamma,\beta}}},
	\end{eqnarray}
	where $C$ is a generic positive constant independent of any function $v$, $N$, $\gamma$, $\beta$.
\end{lem}

\begin{thm}\label{stability}
	For $f \in H_{\omega^{\alpha, 0}}^{r}(I)\cap H^{-\frac{\alpha}{2}}$ with $r \geq -\alpha,$  there exists a unique solution $u_{N}$  to  \eqref{scheme1} such that for  sufficiently large $N$
	\begin{equation*}\label{stability1}
		\left\|u_{N}\right\|_{{\omega^{0,-\alpha}}} \leq C\|f\|_{H_{\omega^{\alpha,0}}^{-\alpha}}.
	\end{equation*}
	Moreover, suppose u solves $\eqref{uweak},$ then we have the  error estimate
	\begin{equation*}\label{conver}
		\left\|u-u_{N}\right\|_{{\omega^{0,-\alpha}}} \leq C N^{-m}\|f\|_{H_{\omega^{\alpha,0}}^{r}},
	\end{equation*}
	where $m$ is the regularity index of $\omega^{0,-\alpha}u$ in $H_{\omega^{0,\alpha}}^{m}(I)$.
\end{thm}
\begin{proof}
	We first prove the well-posedness of the discrete problem \eqref{scheme1}. Note that by \eqref{rderivative} and definition \eqref{projection1} of projection $\pi_N^{\alpha,0}$,
	\begin{equation}\label{xingzhi}
		b(u_N,v-\pi_N^{\alpha,0}v)=\lambda(u_N,v-\pi_N^{\alpha,0}v)_{\omega^{\alpha,0}}.
	\end{equation}
	For $0\not=u_N \in U_N$, by \eqref{infsup}, we have
	\begin{equation}\label{chaifen}
		\frac{1}{C}\|u_N \|_{{\omega^{0,-\alpha}}}\leq \sup\limits_{0\not=v\in H_{\omega^{\alpha, 0}}^{\alpha}}\frac{b(u_N,v)}{\|v\|_{H_{\omega^{\alpha, 0}}^{\alpha}}}= \sup\limits_{0\not=v\in H_{\omega^{\alpha, 0}}^{\alpha}}\frac{b(u_N,v-\pi_N^{\alpha,0}v)}{\|v\|_{H_{\omega^{\alpha, 0}}^{\alpha}}}+\sup\limits_{0\not=v\in H_{\omega^{\alpha, 0}}^{\alpha}}\frac{b(u_N,\pi_N^{\alpha,0}v)}{\|v\|_{H_{\omega^{\alpha, 0}}^{\alpha}}}%:=I+II
	\end{equation}
	By the Cauchy-Schwarz inequality and Lemma \ref{lemma1}, taking $r_1=0$, $r=\alpha$, we obtain
	\begin{equation}\label{i}
		\sup\limits_{0\not=v\in H_{\omega^{\alpha, 0}}^{\alpha}}\frac{\lambda(u_N,v-\pi_N^{\alpha,0}v)_{\omega^{\alpha,0}}}{\|v\|_{H_{\omega^{\alpha, 0}}^{\alpha}}}\leq \sup\limits_{0\not=v\in H_{\omega^{\alpha, 0}}^{\alpha}}\frac{\lambda\|u_N\|_{{\omega^{\alpha,0}}}\|v-\pi_N^{\alpha,0}v\|_{\omega^{\alpha,0}}}{\|v\|_{H_{\omega^{\alpha, 0}}^{\alpha}}}\leq c\lambda N^{-\alpha}\|u_N\|_{{\omega^{0,-\alpha}}}.
	\end{equation}
	
	Using the equivalent norm \eqref{norm2} we have
	\begin{equation}\label{ii}
		\sup\limits_{0\not=v\in H_{\omega^{\alpha, 0}}^{\alpha}}\frac{b(u_N,\pi_N^{\alpha,0}v)}{\|v\|_{H_{\omega^{\alpha, 0}}^{\alpha}}} \leq \sup\limits_{0\not=v\in H_{\omega^{\alpha, 0}}^{\alpha}}\frac{b(u_N,\pi_N^{\alpha,0}v)}{\|\pi_N^{\alpha,0}v\|_{H_{\omega^{\alpha, 0}}^{\alpha}}}=\sup\limits_{0\not=v_N\in V_N}\frac{b(u_N,v_N)}{\|v_N\|_{H_{\omega^{\alpha, 0}}^{\alpha}}}
	\end{equation}
	By \eqref{xingzhi}, substituting \eqref{i} and \eqref{ii}  into \eqref{chaifen}, it follows that
	\begin{equation}\label{zero}
		\sup\limits_{0\not=v_N\in V_N}\frac{b(u_N,v_N)}{\|v_N\|_{H_{\omega^{\alpha, 0}}^{\alpha}}} \geq (\frac{1}{C}-c \lambda N^{-\alpha})\left\|u_{N}\right\|_{{\omega^{0,-\alpha}}}, \quad \forall 0 \neq u_{N} \in U_{N}.
	\end{equation}
	For sufficiently large $N$, the inf-sup condition holds, which  leads to the desired
	conclusion.
	
	Next we present the error estimate of the spectral Petrov-Galerkin method.
	For any $\phi_{N} \in U_{N}$, we have
	\begin{equation*}
		\left\|u-u_{N}\right\|_{{\omega^{0,-\alpha}}} \leq\left\|u-\phi_{N}\right\|_{{\omega^{0,-\alpha}}}+\left\|u_{N}-\phi_{N}\right\|_{{\omega^{0,-\alpha}}},\ \forall\ \phi_{N} \in U_{N}
	\end{equation*}
	By \eqref{zero} and the fact that for any $v_N\in V_{N}$, $b\left(u_{N}-u, v_N\right)=0$,
	\begin{equation*}
		\left\|u_{N}-\phi_{N}\right\|_{{\omega^{0,-\alpha}}} \leq C \sup _{0 \neq v_N \in V_{N}} \frac{b\left(u_{N}-\phi_{N}, v_N\right)}{\left\|v_N\right\|_{H^{\alpha}_{\omega^{\alpha,0}}}}=C \sup _{0 \neq v_N \in V_{N}} \frac{b\left(u-\phi_{N}, v_N\right)}{\left\|v_N\right\|_{H^{\alpha}_{\omega^{\alpha,0}}}} \leq C\left\|u-\phi_{N}\right\|_{{\omega^{0,-\alpha}}}
	\end{equation*}
	Taking $\phi_N=\omega^{0,\alpha}\pi_N^{0,\alpha}(\omega^{0,-\alpha}u)$, by \eqref{sp}, \eqref{hr} and Lemma \ref{lemma2}, we obtain
	\begin{equation*}
		\left\|u-u_{N}\right\|_{{\omega^{0,-\alpha}}} \leq CN^{-m}|\omega^{0,-\alpha}u|_{H_{\omega^{0,\alpha}}^{m}} \leq CN^{-m}\|f\|_{H_{\omega^{\alpha, 0}}^{r}}.
	\end{equation*}

\end{proof}

\section{Numerical example}\label{section6}
In this section, we firstly present a fast iteration algorithm for the linear system produced by the spectral Petrov-Galerkin method,  which is based on  the fast polynomial transform and allows quasilinear computational cost  $O(N\log^2N)$ and linear storage  $O(N)$.  Secondly, we provide three numerical examples to verify the theoretical findings. In Example \ref{ex1}, the smooth source term $f$  is adopted,  and  $f$ with a weak singularity at an interior  in Example \ref{ex2},  $f$ with weak singularity at the origin in Example \ref{ex3}. It is shown from the data that  the numerical results are consistent with the conclusions in Theorem \ref{stability}, Theorem \ref{theorem3.6} and Remark  \ref{re}, correspondingly.  Moreover, a time fractional diffusion problem with the operator $A=-\Delta$ in \eqref{eq:TFDE} is considered in Example \ref{exadd} to verify that our theoretical analysis and numerical method are valid for some time-fractional differential  problems.

In the computation, we take $\lambda=1$. Since exact solutions are unavailable, we  measure the errors  in the following sense:
\begin{equation*}
	E_N=\frac{\|u_N-u_{ref}\|_{\omega^{0,-\alpha}}}{\|u_{ref}\|_{\omega^{0,-\alpha}}},
\end{equation*}
where  $u_{ref}$ is the reference solution computed by the same solver but with a very fine resolution, $u_{ref}:=u_{1024}$.

\subsection{Numerical implementation}\label{sub6}
In this part, we describe the numerical implementation of the spectral Petrov-Galerkin method and present a fast iterative solver inspired by the related discussion in\cite{town} and \cite{hao2}.

Substituting \eqref{form} into the spectral Petrov-Galerkin scheme \eqref{scheme1} and using properties of Jacobi polynomial, we observe
\begin{equation}\label{mm}
	AU=F,
\end{equation}
where $U=(u_0,u_1,\cdots,u_N)^{T}$, $F=(f_0,f_1,\cdots,f_N)^{T}$ with $f_k=(f,Q^{\alpha,0}_k)_{\omega^{\alpha,0}}$, $k=0,1,\cdots,N$. Here $A=S+\lambda M$, where $S$ is a diagonal matrix
$$S=\text{diag}(\lambda_0^{\alpha}h_{0}^{\alpha,0}, \lambda_1^{\alpha}h_{1}^{\alpha,0}, \cdots, \lambda_N^{\alpha}h_{N}^{\alpha,0})^{T}$$
and $M$ is a dense matrix with the following  entries
$$M_{k,n}=\int_0^T \omega^{\alpha,\alpha} Q_n^{0,\alpha}(x)Q_k^{\alpha,0}(x)dx, \ k,n=0,1,\cdots,N.$$
To solve \eqref{mm} directly, the Gauss-Jacobi quadrature rules can be employ to obtain $f_k$ and $M_{k,n}$.  As the system is dense, a direct solver requires $O(N^2)$ storage and $O(N^3)$ computational complexity.

Based on the analysis in the previous sections,  the convergence order and accuracy of the numerical solution are relatively low with rough right-hand function and small $\alpha$.  In this case, we need to take large truncation number  $N$  to improve the accuracy of the numerical solution, which will increase the computational cost of the direct solver significantly.

To overcome this drawback, following the idea in \cite{hao2} we use the fixed-point iteration
\begin{equation}\label{t-iter}
	U^{m+1}=U^{m}+P^{-1}(F-AU^m),
\end{equation}
where the preconditioner $P=S+\lambda Q$ is a diagonal matrix with $Q=\text{diag}(h_0^{\alpha,\alpha}, h_1^{\alpha,\alpha},\cdots, h_N^{\alpha,\alpha})$ and the initial guess can be chosen as the numerical solution obtained by a direct method with $N=8$.  The iterations end when the maximum iteration number $100$ is reached or the condition $\| U^{m+1}- U^{m}\|_{2}/\| U^{m+1}\|_{2}<\epsilon$ is met, where we take $\epsilon=10^{-7}$ and $\|\,\cdot\,\|_{2}$ denotes the classical Euclidean norm.   To contain the information of the reaction term,  $Q$ is introduced  in the current preconditioner  to replace the identity matrix used in the existing literature. The results in Table \ref{ss} indicate that the new preconditioner can gurantee that iteration numbers are independent of $\alpha$ and $N$.

In each iteration, we compute the matrix-vector product  without forming a matrix by applying the fast polynomial transform and the fast matrix-vector product for Toeplitz-dot-Hankel matrix \cite{town}.
Let $Q^{\gamma,\beta}=(Q_0^{\gamma,\beta},Q_1^{\gamma,\beta},\cdots,Q_N^{\gamma,\beta})^T$, $C^{\gamma\rightarrow\sigma,\beta}$ and $C^{\sigma,\delta\rightarrow\beta}$ are lower triangular matrices with constant entries $$(C^{\gamma\rightarrow\sigma,\beta})_{n,k}=c_{n,k}^{\gamma\rightarrow\sigma,\beta},\ (C^{\sigma,\delta\rightarrow\beta})_{n,k}=c_{n,k}^{\sigma,\delta\rightarrow\beta},\ n,k=0,1\cdots,N,$$
where $c_{n,k}^{\gamma\rightarrow\sigma,\beta}$ and $c_{n,k}^{\sigma,\delta\rightarrow\beta}$  satisfy  \cite{askey1}
\begin{align}\label{expan}
	Q_n^{\gamma,\beta}(x)=\sum_{k=0}^{n}c_{n,k}^{\gamma\rightarrow\sigma,\beta}Q_k^{\sigma,\beta}(x)\  \mbox{and} \ Q_n^{\sigma,\delta}(x)=\sum_{k=0}^{n}c_{n,k}^{\sigma,\delta\rightarrow\beta}Q_k^{\sigma,\beta}(x).
\end{align}
That is, $Q^{\gamma,\beta}=C^{\gamma\rightarrow\sigma,\beta}Q^{\sigma,\beta},\ Q^{\sigma,\delta}=C^{\sigma,\delta\rightarrow\beta}Q^{\sigma,\beta}$.  According to the above representation, we have
\begin{eqnarray*}\begin{aligned}
		M&=\int_0^T \omega^{\alpha,\alpha} Q^{\alpha,0}(Q^{0,\alpha})^T dx\\
		&=\int_0^T \omega^{\alpha,\alpha}  C^{\alpha,0\rightarrow\alpha}Q^{\alpha,\alpha}(C^{0\rightarrow\alpha,\alpha}Q^{\alpha,\alpha})^Tdx\\
		&=\int_0^T \omega^{\alpha,\alpha}C^{\alpha,0\rightarrow\alpha}Q^{\alpha,\alpha}(Q^{\alpha,\alpha})^T(C^{0\rightarrow\alpha,\alpha})^Tdx\\
		&=C^{\alpha,0\rightarrow\alpha}H^{\alpha}(C^{0\rightarrow\alpha,\alpha})^T,
\end{aligned}\end{eqnarray*}
where $H^{\alpha}=\text{diag}(h_0^{\alpha,\alpha},h_1^{\alpha,\alpha},\cdots,h_N^{\alpha,\alpha})$. Denote $u_N=t^{\alpha}\hat{u}_N$, by using \eqref{expan} we obtain
\begin{align}
	\hat{u}_N=\sum\limits_{n=0}^{N}u_nQ_n^{0,\alpha}(x)=(Q^{0,\alpha})^TU=(Q^{\alpha,\alpha})^T(C^{0\rightarrow\alpha,\alpha})^TU:=(Q^{\alpha,\alpha})^TU^{\alpha,\alpha}=\sum\limits_{n=0}^{N}u_n^{\alpha,\alpha}Q_n^{\alpha,\alpha}(x)\label{UU}
\end{align}
where $U^{\alpha,\alpha}=(u_0^{\alpha,\alpha},u_1^{\alpha,\alpha},\cdots,u_N^{\alpha,\alpha})^T$ can be observed by the fast polynomial transform \cite{town}.
Note that from \cite{askey1}

$$c_{n,k}^{\alpha,0\rightarrow\alpha}=\frac{(-1)^{n-k}(2k+2\alpha+1)\Gamma(k+2\alpha+1)\Gamma(n-k-\alpha)\Gamma(n+k+\alpha+1)}{\Gamma(-\alpha)\Gamma(k+\alpha+1)\Gamma(n-k+1)\Gamma(n+k+2\alpha+2)}.$$

Thus, for the connection coefficients matrix $C^{\alpha,0\rightarrow\alpha}$ we can write
\begin{align}
	C^{\alpha,0\rightarrow\alpha}=(T\circ H) D,\label{M}
\end{align}
where $D$ is diagonal matrix, $T$ is a  Toeplitz matrix, $H$ is a Hankel matrix and `$\circ$' is the Hadamard matrix product, i.e., entrywise multiplication between two matrices. Specifically, for $0\leq k\leq n\leq N$,
\begin{align}
	(D)_{k,k}&=\frac{(2k+2\alpha+1)\Gamma(k+2\alpha+1)}{\Gamma(-\alpha)\Gamma(k+\alpha+1)},\nonumber\\
	(T)_{n,k}&=(-1)^{n-k}\frac{\Gamma(n-k-\alpha)}{\Gamma(n-k+1)},\ (H)_{n,k}=\frac{\Gamma(n+k+\alpha+1)}{\Gamma(n+k+2\alpha+2)}.\nonumber
\end{align}
Using \eqref{UU} and \eqref{M}, it follows that
$$MU=C^{\alpha,0\rightarrow\alpha}H^{\alpha}(C^{0\rightarrow\alpha,\alpha})^TU=C^{\alpha,0\rightarrow\alpha}H^{\alpha}U^{\alpha,\alpha}=(T\circ H)DH^{\alpha}U^{\alpha,\alpha}:=(T\circ H)v.$$
Now we are in the position to use the pivoted Cholesky algorithm approximating the Hankel matrix $H$  by a low rank matrix  and  the fast Toeplitz matrix-vector product.  As shown in \cite{town}, the present fast iteration solver allows the  quasilinear computational cost  $O(N\log^2N)$ and the linear storage  $O(N)$.

\subsection{Numerical results}
In this part we show some examples where different regularities of $f$ are considered.
\begin{ex}\label{ex1} Take $T=1$, $f=sin(t-\frac{T}{2})$ in \eqref{FIVP}. Note that $f\in H^{\infty}_{\omega^{\alpha,0}}(I)$.
\end{ex}

By Theorem \ref{theorem3.6}, as the source term $f$ is analytic, we have $\omega^{0,-\alpha}u \in H^{3\alpha+1-\epsilon}_{\omega^{0,\alpha}}(I)$. From Theorem \ref{conver}, the convergence order is expected to be $3\alpha+1-\epsilon$.  In Table \ref{Table1}, we test convergence orders by a direct solver \eqref{mm}. In Table \ref{ss}, we check the performance of the fast iteration solver with the reference solution  $u_{ref}=u_{2^{14}}$. We tabulate the convergence orders and relative errors of numerical solutions of the spectral Petrov-Galerkin method in $L^2_{\omega^{0,-\alpha}}$-norm for different values of $\alpha$,  %in Table \ref{Table1} by a direct solver \eqref{mm} and in Table \ref{ss} by the fast iteration solver,
from which the numerical results confirm our theoretical findings in Theorems \ref{theorem3.6} and  \ref{conver}.  The number of iterations is not greater than 10  and independent of $\alpha$ and $N$.

Moreover, we list the convergence orders and relative errors of our numerical solutions in standard $L^2$-norm in Table \ref{Table1.1}. It is shown  that the convergence order and accuracy of the numerical solutions in $L^2$-norm are higher than its in weighted  $L^2$-norm.

\begin{table}[htbp]
	\scriptsize
	\caption{Convergence orders and errors of the spectral Petrov-Galerkin method for Example 6.1 with $f=sin(t-\frac{T}{2})$. The expected convergence order is $3\alpha+1-\epsilon$ in the $L^2_{\omega^{0,-\alpha}}$-norm (Theorem \ref{theorem3.6} and Theorem \ref{stability}).}
	\label{Table1}
	\centering
	\begin{tabular}{|c|cc| cc|cc|cc|}
		%{@{}llrrrrrrrr@{}}
		\hline
		\multirow{2}{*}{$N$} & \multicolumn{2}{c|}{$\alpha$=0.2} &  \multicolumn{2}{c|}{$\alpha$=0.4} &  \multicolumn{2}{c|}{$\alpha$=0.6} & \multicolumn{2}{c|}{$\alpha$=0.8}\\
		\cline{2-9}
		&$E_N$ &  rate&   $E_N$ &  rate   &$E_N$ &  rate  & $E_N$ &  rate \\
		\hline
		
		32       &  1.29e-03   &    *       &      2.92e-04&    *     &  3.49e-05     &    *       &      2.60e-06&      *            \\
		64       &   4.70e-04  &    1.46    &     6.72e-05   &    2.12  &   5.25e-06   &    2.73    &  2.63e-07   &    3.31        \\
		128       &   1.67e-04   &    1.50   &   1.50e-05   &   2.16  &    7.74e-07  &    2.76     &  2.58e-08    &    3.35     \\
		256       &   5.75e-05  &    1.53    &   3.32e-06    &    2.18  &      1.13e-07&    2.78    &   2.48e-09  &    3.37      \\
		512      &    1.87e-05  &    1.62    &   7.12e-07  &    2.22  &     1.61e-08&    2.81   &  2.36e-10 &    3.39       \\
		\hline
		Expected order&  & 1.6 & &2.2 & &2.8 & & 3.4 \\
		\hline
	\end{tabular}
\end{table}

\begin{table}[htbp]
	\scriptsize
	\caption{Tests of the proposed fast iterative solver in convergence and computational time for Example 6.1 with $f=sin(t-\frac{T}{2})$. The estimated convergence order is $3 \alpha+1$ in the $L^2_{\omega^{0,-\alpha}}$-norm. Here `iter' represents the iteration number and `CPU(s)' stands for the computational time measured in seconds  (Theorem \ref{theorem3.6} and Theorem \ref{stability}).}
	\label{ss}
	\centering
	\begin{tabular}{|c|cccc| cccc|}
		%{@{}llrrrrrrrr@{}}
		\hline
		\multirow{2}{*}{$N$} & \multicolumn{4}{c|}{$\alpha$=0.2} &  \multicolumn{4}{c|}{$\alpha$=0.4} \\
		\cline{2-9}
		&$E_N$ &  rate&    iter &CPU(s)& $E_N$ &  rate & iter &CPU(s) \\
		\hline
		
		512       & 2.59e-06     &          & 9     &     0.03     & 8.31e-08     &          &10     &     0.03    \\
		1024       & 8.78e-07     &     1.56     & 9     &     0.06     & 1.82e-08     &     2.19     & 10     &     0.06    \\
		2048       & 2.96e-07     &     1.57     & 9     &     0.10     & 3.96e-09     &     2.20     & 10     &     0.11    \\
		4096       & 9.86e-08     &     1.58     & 9     &     0.22     & 8.63e-10     &     2.20     & 10     &     0.23    \\
		\hline
		Expected order&  & 1.6 & & & &2.2 & &  \\
		\hline
		\hline
		
		\multirow{2}{*}{$N$} & \multicolumn{4}{c|}{$\alpha$=0.6} &  \multicolumn{4}{c|}{$\alpha$=0.8} \\
		\cline{2-9}
		&$E_N$ &  rate  & iter  &CPU(s)  &$E_N$ &  rate   & iter  &CPU(s) \\
		\hline
		512       & 1.56e-09     &          & 10     &     0.03    & 1.88e-11     &          &10     &     0.03    \\
		1024       & 2.25e-10     &     2.80     & 10     &     0.06     & 1.79e-12     &     3.39     & 10     &     0.06    \\
		2048       & 3.23e-11     &     2.80     & 10     &     0.11     & 1.70e-13     &     3.39     & 10     &     0.11    \\
		4096       & 4.65e-12     &     2.80     & 10     &     0.23     & 1.61e-14     &     3.40     & 10     &     0.23    \\
		\hline
		Expected order&  & 2.8 & & & &3.4 & &  \\
		\hline
	\end{tabular}
\end{table}

\begin{table}[htbp]
	\scriptsize
	\caption{Convergence orders and errors of the spectral Petrov-Galerkin method for Example 6.1 with $f=sin(t-\frac{T}{2})$. The expected convergence order is $3\alpha+1-\epsilon$ in the $L^2$-norm (higher than that in Table 1).}
	\label{Table1.1}
	\centering
	\begin{tabular}{|c|cc| cc|cc|cc|}
		%{@{}llrrrrrrrr@{}}
		\hline
		\multirow{2}{*}{$N$} & \multicolumn{2}{c|}{$\alpha$=0.2} &  \multicolumn{2}{c|}{$\alpha$=0.4} &  \multicolumn{2}{c|}{$\alpha$=0.6} & \multicolumn{2}{c|}{$\alpha$=0.8}\\
		\cline{2-9}
		&$E_N$ &  rate&   $E_N$ &  rate   &$E_N$ &  rate  & $E_N$ &  rate \\
		\hline
		32 &   7.53e-05  & *                &  9.86e-06     & *       &    7.99e-07  & * &    4.39e-08  & *   \\
		64       &   2.42e-05  &    1.64   &      1.81e-06&    2.45  &   9.04e-08   &    3.14 &   3.25e-09   &    3.76 \\
		128       &    7.53e-06  &    1.68   &    3.19e-07 &    2.50  &  9.86e-09   &    3.20 &    2.30e-10 &    3.82 \\
		256       &    2.30e-06  &    1.71    &   5.52e-08   &    2.53  & 1.05e-09      &    3.23 & 1.59e-11     &    3.86\\
		512       &    6.95e-07  &    1.73   &  9.57e-09      &    2.53     & 1.12e-10 &3.24   & 1.09e-12  &3.87 \\
		\hline
	\end{tabular}
\end{table}

\begin{ex}\label{ex2}  Take $T=1$, $f=|\sin(t-\frac{T}{2})|$ in \eqref{FIVP}. Note that $f\in H^{1.5-\epsilon}_{\omega^{\alpha,0}}(I)$.
\end{ex}

\begin{table}[htbp]
	\scriptsize
	\caption{Convergence orders and errors of the spectral Petrov-Galerkin method for Example 6.2 with $f=|\sin(t-\frac{T}{2})|$. The expected convergence order is $\min(3\alpha+1-\epsilon,1.5+\alpha-\epsilon)$ (Theorem \ref{theorem3.6} and Theorem \ref{stability}).}
	\label{Table2}
	\centering
	\begin{tabular}{|c|cc| cc|cc|cc|}
		%{@{}llrrrrrrrr@{}}
		\hline
		\multirow{2}{*}{$N$}& \multicolumn{2}{c|}{$\alpha$=0.1} &  \multicolumn{2}{c|}{$\alpha$=0.2} &  \multicolumn{2}{c|}{$\alpha$=0.4} & \multicolumn{2}{c|}{$\alpha$=0.6} \\
		\cline{2-9}
		&$E_N$ &  rate&   $E_N$ &  rate   &$E_N$ &  rate  & $E_N$ &  rate \\
		\hline
		32       &5.51e-03     &    *          &  6.32e-03      &    *             & 2.98e-03     &    *    &   1.20e-03     &    *  \\
		64       & 2.35e-03     &  1.23  & 2.50e-03      &   1.34  & 8.22e-04     &    1.86    & 3.04e-04     &    1.98   \\
		128       & 1.03e-03     &    1.19     &9.67e-04     &   1.37    &2.24e-04     &    1.88  & 7.51e-05     &    2.02    \\
		256      & 4.54e-04    &     1.18   & 3.62e-04      &    1.42  & 6.10e-05     &    1.88  &  1.81e-05     &    2.05\\
		512       & 1.88e-04     &     1.27   &  1.26e-04      &   1.53  & 1.61e-05     &    1.92    & 4.22e-06     &    2.11 \\
		\hline
		Expected order&  & 1.3 & &1.6 & &1.9 & & 2.1 \\
		\hline
	\end{tabular}
\end{table}

By Theorem \ref{theorem3.6}, we have $\omega^{0,-\alpha}u \in H^{\min(3\alpha+1-\epsilon, 1.5+\alpha-\epsilon)}_{\omega^{0,\alpha}}(I)$. According to Theorem \ref{conver}, we expect the  convergence order of numerical solutions  is $\min(3\alpha+1-\epsilon, 1.5+\alpha-\epsilon)$ in $L^2_{\omega^{0,-\alpha}}$ norm. In Table \ref{Table2}, we test the convergence orders and errors  for different $\alpha$, and it can be observed that the convergence order depends on the value of $\alpha$, that is, $\min(3\alpha+1-\epsilon, 1.5+\alpha-\epsilon)$, which coincides with the theoretical prediction and the regularity analysis in Theorems \ref{theorem3.6} and \ref{stability}.

\begin{ex}\label{ex3} Take $T=2$, $f=t^{\sigma}e^t$, $\sigma>0$, in \eqref{FIVP}. Note that $f(0)=0$.
	
	\begin{table}[htbp]
		\scriptsize
		\caption{Convergence orders and errors of the spectral Petrov-Galerkin method for Example 6.3 with $f=t^{0.3}e^t$. The expected convergence order is $\min(3\alpha+3-\epsilon,1.6+\alpha-\epsilon)$ in the $L^2_{\omega^{0,-\alpha}}$-norm (Theorem \ref{theorem3.6}, Theorem \ref{stability} and Remark \ref{re}).}
		\label{Table3}
		\centering
		\begin{tabular}{|c|cc| cc|cc|cc|}
			%{@{}llrrrrrrrr@{}}
			\hline
			\multirow{2}{*}{$N$} & \multicolumn{2}{c|}{$\alpha$=0.2} &  \multicolumn{2}{c|}{$\alpha$=0.4} &  \multicolumn{2}{c|}{$\alpha$=0.6} & \multicolumn{2}{c|}{$\alpha$=0.8}\\
			\cline{2-9}
			&$E_N$ &  rate&   $E_N$ &  rate   &$E_N$ &  rate  & $E_N$ &  rate \\
			\hline
			8       & 1.23e-03     &      * & 1.09e-03     &      *    & 8.88e-04     &      *  & 6.57e-04     &      * \\
			16       & 4.40e-04     &    1.48 & 3.45e-04     &    1.66    & 2.37e-04     &    1.91& 1.51e-04     &    2.12  \\
			32       & 1.46e-04     &    1.59& 9.82e-05     &    1.81    & 5.70e-05     &    2.06 & 3.15e-05     &    2.26 \\
			64       & 4.60e-05     &    1.66 & 2.63e-05     &    1.90    & 1.30e-05     &    2.13    & 6.28e-06     &    2.33\\
			128       & 1.41e-05     &    1.71 & 6.80e-06     &    1.95         & 2.90e-06     &    2.17  & 1.22e-06     &    2.36     \\
			\hline
			Expected order&  & 1.8 & &2.0 & &2.2 & & 2.4 \\
			\hline
		\end{tabular}
	\end{table}
	
	\begin{table}[htbp]
		\scriptsize
		\caption{Convergence orders and errors of the spectral Petrov-Galerkin method for Example 6.3 with $f=te^t$. The expected convergence order is $3\alpha+3-\epsilon$ in the $L^2_{\omega^{0,-\alpha}}$-norm (Theorem \ref{theorem3.6}, Theorem \ref{stability} and Remark \ref{re}).}
		\label{Table3.1}
		\centering
		\begin{tabular}{|c|cc| cc|cc|cc|}
			%{@{}llrrrrrrrr@{}}
			\hline
			\multirow{2}{*}{$N$} & \multicolumn{2}{c|}{$\alpha$=0.2} &  \multicolumn{2}{c|}{$\alpha$=0.4} &  \multicolumn{2}{c|}{$\alpha$=0.6} & \multicolumn{2}{c|}{$\alpha$=0.8}\\
			\cline{2-9}
			&$E_N$ &  rate&   $E_N$ &  rate   &$E_N$ &  rate  & $E_N$ &  rate \\
			\hline
			8       & 1.65e-05     &      *    & 1.62e-05     &      * & 8.21e-06     &      *  & 2.16e-06     &      *\\
			16       & 2.06e-06     &    3.01   & 1.34e-06     &    3.59& 4.19e-07     &    4.29& 7.14e-08     &    4.92  \\
			32       & 2.19e-07     &    3.23     & 9.08e-08     &    3.89& 1.79e-08     &    4.55& 2.05e-09     &    5.12\\
			64       & 2.13e-08     &    3.36   & 5.51e-09     &    4.04  & 6.99e-10     &    4.68 & 5.38e-11     &    5.25\\
			128       & 1.96e-09     &    3.44   & 3.17e-10     &    4.12 & 2.62e-11     &    4.74  & 1.34e-12     &    5.32\\
			\hline
			Expected order&  & 3.6 & &4.2 & &4.8 & & 5.4 \\
			\hline
		\end{tabular}
	\end{table}
	
\end{ex}
In Table \ref{Table3}, we take $\sigma=0.3$ and test the convergence order of the numerical solution  in  the  $L^2_{\omega^{0,-\alpha}}$-norm. In this case, $f$ has weak singularity at the origin and $f\in H^{2\sigma+1-\epsilon}_{\omega^{\alpha,0}}(I)$, $\epsilon>0$. By Remark \ref{re}, $\omega^{0,-\alpha}u \in H^{\alpha+\min(2\alpha+3,1.6)-\epsilon}_{\omega^{0,\alpha}}(I)$.  From Theorem \ref{conver}, the expected convergence order is $\alpha+\min(2\alpha+3,1.6)-\epsilon$  in  the  $L^2_{\omega^{0,-\alpha}}$-norm, which is consistent with our numerical results in Table \ref{Table3}.

In Table \ref{Table3.1}, taking $\sigma=1$ and $f\in  H^{\infty}_{\omega^{\alpha,0}}(I)$, we test the convergence order of the numerical solution  in  the  $L^2_{\omega^{0,-\alpha}}$-norm. By Remark \ref{re} and Theorem \ref{conver}, the expected convergence order is $3\alpha+3-\epsilon$, $\epsilon>0$ in $L^2_{\omega^{0,-\alpha}}-$norm as observed in Table \ref{Table3.1}. Compared with Example \ref{ex1} for smooth enough source term $f$, the convergence order is increased from $3\alpha+1-\epsilon$ to $3\alpha+3-\epsilon$ when the compatible condition, $f(0)=0$, is satisfied.

\begin{ex}\label{exadd}
	Take the positive definite operator $A=-\Delta$, $T=1$, $\Omega=(0,1)$, $g(x,t)=t^{\sigma}e^t\sin\pi x$ in \eqref{eq:TFDE}. 
\end{ex}
Let the space step size $h=1/M$ with $M\in \mathbb{N}^+$, $x_i=ih$, $i=0,1,\cdots,M$, and $u^h(x_i,t)$ is the difference approximation of $u(x_i,t)$. Without loss of generality, we use the second-order central difference scheme for spatial discretization and 
obtain the following FIVP
\begin{eqnarray*}\label{semi}\begin{aligned}
		&u^h(x_i,t)-\delta_{x}^2u^h(x_i,t)=g(x_i,t),\ 1\leq i\leq M-1,\\
		&u^h(x_i,0)=0, \ 1\leq i\leq M-1,\\
		&u^h(x_0,t)=0,\ u^h(x_M,t)=0,\ t\in I,
\end{aligned}\end{eqnarray*}
where $ \delta_{x}^2u^h(x_i):=\frac{1}{h^2}(u^h(x_{i+1})-2u^h(x_i)+u^h(x_{i-1})).$ 
Then by applying the spectral Petrov-Galerkin method in this paper to temporal discretization, we obtain the fully discretized scheme of problem \eqref{eq:TFDE} 
\begin{equation}\label{full}
	(u_N^h(x_i),{}_{t}D_T^{\alpha}(\omega^{\alpha,0}v_N))-(\delta_{x}^2u_N^h(x_i),v_N)_{\omega^{\alpha,0}}=(g(x_i),v_N)_{\omega^{\alpha,0}},\ \forall v_N\in V_N,\ 1\leq i\leq M-1,
\end{equation}
where the numerical solution $u_N^h(t)=(u_N^h(x_0),u_N^h(x_1),\cdots,u_N^h(x_{M}))^T,$ and $u_N^h(x_i):=t^{\alpha}\sum\limits_{n=0}^{N}u_{in}Q_n^{0,\alpha}(t)$ is the approximation of $u^h(x_i,t)$. 
Taking $v_N=Q_j^{\alpha,0}$, $j=0,1,\cdots,N$, in the scheme \eqref{full}  and using properties of Jacobi polynomial, we observe
\begin{equation}\label{spaced}
	\bf{AU=G},
\end{equation}
where ${\bf{A}}=\rm{kron}(I_1,S)-\rm{kron}(I_2,M)$, ${\bf{U}}=(U_1,U_2,\cdots,U_{M-1})^T$, and ${\bf{G}}=(G_1,G_2,\cdots,G_{M-1})^T$ with
$$U_i=(u_{i0},u_{i1},\cdots,u_{iN})^T,\ G_i=(g_{i0},g_{i1},\cdots,g_{iN})^T.$$ Here $I_1$ is a $(M-1)$-dimensional identity matrix, $I_2$ is a tridiagonal Toeplitz matrix produced by spacial discretization, $\rm{kron}(\cdot,\cdot)$ denote the Kronecker product between two matrices, and $g_{ij}=(g(x_i),Q_{j}^{\alpha,0})_{\omega^{\alpha,0}}$. 

Similar to \eqref{t-iter}, we use the fix point iteration to solve \eqref{spaced}
\begin{equation*}
	{\bf U}^{m+1}={\bf U}^{m}+{\bf P}^{-1}({\bf G}-{\bf A}{\bf U}^m).
\end{equation*}
Note that the coefficient matrix ${\bf A}$ is a block tridiagonal Toeplitz matrix, and its main information is concentrated on the diagonal ($\alpha<0.5$, especially for $\alpha$ close to 0) or tridiagonal ($\alpha\geq 0.5$, especially for $\alpha$ close to 1) of each block. Thus we take the preconditioner $P={\rm kron}(I_1,S)+{\rm kron}(I_2,M_1)$ for relatively small $\alpha$  and $P={\rm kron}(I_1,S)+{\rm kron}(I_2,M_2)$ for relatively large $\alpha$, where $M_1$ and $M_2$ are $(N+1)\times(N+1)$ matrices composed of diagonal entries and tridiagonal entries of matrix $M$, respectively.

In each iteration, we compute the matrix-vector product ${\bf A}{\bf U}^m$ and right hand side ${\bf G}$ by applying the matrix-free implementation in subsection \ref{sub6}, and the errors are computed as follows:
\begin{equation*}
	E_N^h=(h\sum\limits_{i=1}^{M-1}\|u_{ref}^h(x_i)-u_N^h(x_i)\|^2_{\omega^{0,-\alpha}})^{1/2}.
\end{equation*}

It is noted that $g(\cdot,t)\in H^{\infty}_{\omega^{\alpha,0}}(I)$  when $\sigma=0$, and $g(\cdot,t)\in H^{2\sigma+1-\epsilon}_{\omega^{\alpha,0}}(I)$ with $g(\cdot,0)=0$ when $\sigma=0.3$. By Theorem \ref{theorem3.6} and Remark \ref{re}, $\omega^{0,-\alpha}u \in H^{3\alpha+1-\epsilon}_{\omega^{0,\alpha}}(I)$ for $\sigma=0$ and $\omega^{0,-\alpha}u \in H^{\min(3\alpha+3,2\sigma+\alpha+1)-\epsilon}_{\omega^{0,\alpha}}(I)$ for $\sigma=0.3$. Thus, by Theorem \ref{conver}, the temporal convergence order should be $3\alpha+1-\epsilon$ for $\sigma=0$, and $\min(3\alpha+3,2\sigma+\alpha+1)-\epsilon$ for $\sigma=0.3$.  We tabulate the convergence orders and errors of the numerical solution for various values of $N$ and $\alpha$ with $M=2^{10}$, $\sigma=0$ (Table \ref{Tableadd}) and $\sigma=0.3$ (Table \ref{Tableadd2}).   We observe that the numerical results in Tables \ref{Tableadd}-\ref{Tableadd2} illustrate the optimal convergence order. To some extent, it exhibits that our theoretical analysis and numerical method for the FIVP \eqref{FIVP} can provide an effective framework for solving the time fractional diffusion equations.

\begin{table}[htbp]
	\scriptsize
	\caption{Convergence orders and errors of the spectral Petrov-Galerkin method for the time-fractional differential equation \eqref{eq:TFDE} in Example \ref{exadd} with $g=e^t\sin\pi x$. The expected temporal convergence order is $3\alpha+1-\epsilon$ in the $L^2_{\omega^{0,-\alpha}}$-norm (Theorem \ref{theorem3.6}, Theorem \ref{stability}).}
	\label{Tableadd}
	\centering
	\begin{tabular}{|c|cc| cc||c|cc|cc|}
		%{@{}llrrrrrrrr@{}}
		\hline
		\multirow{2}{*}{$N$} & \multicolumn{2}{c|}{$\alpha$=0.2} &  \multicolumn{2}{c||}{$\alpha$=0.4} &\multirow{2}{*}{$N$} &  \multicolumn{2}{c|}{$\alpha$=0.6} & \multicolumn{2}{c|}{$\alpha$=0.8}\\
		\cline{2-5}\cline{7-10}
		&$E_N$ &  rate&   $E_N$ &  rate&   &$E_N$ &  rate  & $E_N$ &  rate \\
		\hline
		256       & 2.37e-05    &      *    & 1.54e-06     &      * &32& 3.15e-06     &      *  & 4.32e-07     &      *\\
		512       & 9.88e-06     &    1.26   & 3.49e-07          &    2.15& 64&4.70e-07     &    2.74& 4.13e-08     &    3.39  \\
		1024       & 3.96e-06     &    1.32     & 7.75e-08      &    2.17&128& 6.83e-08     &    2.78& 3.98e-09     &    3.38\\
		2048       &  1.53e-06     &    1.37   & 1.70e-08          &    2.18  &256& 9.87e-09     &    2.79 & 3.81e-10     &    3.38\\
		4096       & 5.70e-07     &    1.43   & 3.73e-09         &    2.19 &512& 1.42e-09     &    2.80  & 3.64e-11     &    3.39\\
		\hline
		Expected order&  & 1.6 & &2.2 &Expected order& &2.8 & & 3.4 \\
		\hline
	\end{tabular}
\end{table}

\begin{table}[htbp]
	\scriptsize
	\caption{Convergence orders and errors of the spectral Petrov-Galerkin method for the time-fractional differential equation \eqref{eq:TFDE} in Example \ref{exadd} with $g=t^{0.3}e^t\sin\pi x$. The expected temporal convergence order is $\min(3\alpha+3-\epsilon,1.6+\alpha-\epsilon)$  in the $L^2_{\omega^{0,-\alpha}}$-norm (Theorems \ref{theorem3.6}, \ref{stability} and Remark \ref{re}).}
	\label{Tableadd2}
	\centering
	\begin{tabular}{|c|cc| cc||c|cc|cc|}
		%{@{}llrrrrrrrr@{}}
		\hline
		\multirow{2}{*}{$N$} & \multicolumn{2}{c|}{$\alpha$=0.2} &  \multicolumn{2}{c||}{$\alpha$=0.4} &\multirow{2}{*}{$N$} &  \multicolumn{2}{c|}{$\alpha$=0.6} & \multicolumn{2}{c|}{$\alpha$=0.8}\\
		\cline{2-5}\cline{7-10}
		&$E_N$ &  rate&   $E_N$ &  rate&   &$E_N$ &  rate  & $E_N$ &  rate \\
		
		\hline
		128      & 9.06e-06     &      *    & 7.78e-06     &      *&32 & 6.28e-05     &      *  & 3.72e-05     &      *\\
		256       & 3.14e-06     &    1.53   & 2.16e-06     &    1.85&64& 1.58e-05     &    1.99& 7.95e-06     &    2.23  \\
		512       & 1.06e-06     &    1.57     & 5.72-07     &    1.91&128& 3.72e-06     &    2.09& 1.62e-06     &    2.30\\
		1024       & 3.46e-07     &    1.61   & 1.47e-07    &   1.95 &256  & 8.42e-07     &    2.15 & 3.19e-07    &    2.34\\
		2048       & 1.08e-07     &    1.68   & 3.67e-08     &    2.00 &512& 1.87e-07     &    2.17  & 6.21e-08     &    2.36\\
		\hline
		Expected order&  & 1.8 & &2.0 &Expected order& &2.2 & & 2.4 \\
		\hline
	\end{tabular}
\end{table}

\section{Conclusion}
In this paper, we have studied a spectral Petrov-Galerkin method for  fractional initial value problems with Caputo fractional derivative,  which can provide an efficient framework for dealing with some time-fractional differential problems.
To capture the singularity of the solution at the origin, we have analyzed the regularity of the solution in weighted Sobolev space by using a bootstrapping technique.
When $\lambda>0$, if the regularity index of $f$ is $r\geq-\alpha$ in weighted Sobolev space, the regularity index of $t^{-\alpha}u$ is $\alpha+\min(r,2\alpha+1-\epsilon)$, where $\epsilon>0$ is an arbitrary small number and $\alpha$ is the fractional order. %And if $f$ has weak singularity at the origin, the regularity index of  $t^{-\beta}u$ is  $r+\alpha$, where $r$ is the regularity index of $t^{\alpha-\beta}f$ in weighted Sobolev space with negative parameter $\beta$.
Moreover,  the stability and an optimal error estimate of the spectral Petrov-Galerkin method have been established in a weighted $L^2$-norm.

Numerical examples verify the theoretical findings for the FIVPs \eqref{FIVP} with  source terms which are smooth enough  (see Example \ref{ex1}),  with weak singularity at an interior (see Example \ref{ex2}) and  with weak singularity at the origin (see Example \ref{ex3}).  That is,  the tested convergence orders   in the weighted $L^2$-norm are correspondingly consistent with the expected ones in Theorem \ref{theorem3.6}, Theorem \ref{stability} and Remark \ref{re}. In additional, it is exhibited that our theoretical analysis and numerical method are also applicable for a large class of problems including the time-fractional diffusion equations (see Example \ref{exadd}). 

To implement the presented  spectral Petrov-Galerkin method efficiently, we have also presented a fast iteration algorithm for solving the resulting linear system, which reduces the computational cost from  $O(N^3)$ to  $O(N\log^2N)$ and the storage from $O(N^2)$ to $O(N)$, compared with the direct solver.

\section*{Acknowledgment}
This work was partially supported by the National Natural Science Foundation of China (No. 12071073 and No. 11671083).

\end{document}